\documentclass[lettersize,journal]{IEEEtran}
\usepackage{amsmath,amsfonts,bm}
\usepackage{array}
\usepackage{listings}
\usepackage{adjustbox}
\usepackage[group-digits=true,group-minimum-digits=4]{siunitx}
\usepackage[caption=false,font=normalsize,labelfont=sf,textfont=sf]{subfig}
\usepackage{tikz,pgfplots}
\usepackage{textcomp}
\usepackage{stfloats}
\usepackage{url}
\usepackage{verbatim}
\usepackage{graphicx}
\usepackage{acronym}
\usepackage{caption}
\usepackage{float}
\usepackage{comment}
\usepackage{soul}
\usepackage{color}
\usepackage{mathtools}
\usepackage{graphicx}
\usepackage{balance}
\usepackage{multicol}
\usepackage{longtable}
\usepackage{amssymb}
\usepackage{booktabs, makecell, tabularx}
\usepackage{longtable}
\usepackage[section]{placeins}
\usepackage{nomencl}
\usepackage{footmisc}
\usepackage{cite}
\usepackage{xcolor}
\usepackage{caption}
\usepackage{setspace}
\usepackage{amsthm}
\usepackage{academicons}
\usepackage[hidelinks]{hyperref}
\usepackage[acronym]{glossaries}
\usepackage{mathtools, cuted}
\usepackage{lipsum}
\usepackage{balance}
\usepackage[font={footnotesize},labelfont={footnotesize}]{caption}
\usepackage{xpatch}
\newtheorem{lemma}{\textbf{Lemma}}
\makeatletter
\xpatchcmd{\proof}{\@addpunct{.}}{\@addpunct{:}}{}{}
\makeatother

\begin{document}

\title{Enhancing Non-Terrestrial Network Performance with Free Space Optical Links and Intelligent Reflecting Surfaces}

\author{Shunyuan Shang,~Emna~Zedini,~\IEEEmembership{Member,~IEEE,}~and~Mohamed-Slim~Alouini,~\IEEEmembership{Fellow,~IEEE} 
\thanks{(\textit{Corresponding author: Shunyuan Shang})}
\thanks{S. Shang and M.-S. Alouini are with the Computer, Electrical, and Mathematical Science and Engineering (CEMSE) Division, King Abdullah University of Science and Technology (KAUST), Thuwal, Makkah Province, Saudi Arabia (e-mail: shunyuan.shang@kaust.edu.sa; slim.alouini@kaust.edu.sa).}
    \thanks{E. Zedini is with the College of Innovation and Technology, University of Michigan-Flint, Flint, MI, USA (e-mail: ezedini@umich.edu).}}

\markboth{}%
{Shell \MakeLowercase{\textit{et al.}}: Enhancing Non-Terrestrial Network Performance with Free Space Optical Links and Intelligent Reflecting Surfaces}

\maketitle

\begin{abstract}
The integration of non-terrestrial networks (NTNs), which include high altitude platform (HAP) stations and intelligent reflecting surfaces (IRS) into communication infrastructures has become a crucial area of research to address the increasing requirements for connectivity and performance in the post-5G era. This paper presents a comprehensive performance study of a new NTN architecture, which enables communication from the optical ground station (OGS) to end users through the utilization of HAP and terrestrial IRS nodes. In this configuration, the HAP acts as an amplify-and-forward (AF) relay terminal between the free-space optical (FSO) link and the RF links.
 Specifically, the RF links are modeled using the Shadowed Rician and the generalized Nakagami-$m$ models, where the FSO link is characterized by the Gamma-Gamma distribution with generalized pointing errors. The FSO system operates under either intensity modulation with direct detection or heterodyne detection. Using the mixture Gamma model, we approximate the non-centered chi-square distribution that describes the total fading of the RF link, and we assess the performance of the end-to-end system by analyzing the ergodic capacity, the average bit-error rate (BER), and the outage probability, calculated using the bivariate Fox-H function. We also provide simple asymptotic expressions for the average BER and the outage probability at high signal-to-noise ratio (SNR). Finally, the proposed analysis is validated with numerical and Monte-Carlo simulation results, showing an exact match.

\end{abstract}

\begin{IEEEkeywords}
Non-terrestrial networks (NTNs), high altitude platform (HAP), intelligent reflecting surfaces (IRS), free-space optical (FSO) links, Gamma-Gamma, generalized pointing error, Shadowed Rician, mixture Gamma distribution.

\end{IEEEkeywords}

\IEEEpeerreviewmaketitle

\section{Introduction}
As the demand for data connectivity continues to surge due to the widespread use of digital technologies and connected devices, the focus on advancing communication technologies beyond 5G towards 6G has intensified \cite{jiang2021road,qadir2023towards,shen2023five,quy2023innovative,underrei}. 6G is anticipated to revolutionize wireless communication by introducing cutting-edge technologies that operate at Terahertz and optical frequencies, as well as innovative network architectures \cite{jiang2024terahertz,tomkos2020toward}. A key aspect of this evolution involves the exploration of non-terrestrial networks (NTNs), which aim to enhance traditional ground-based infrastructures by using aerial platforms like satellites, unmanned aerial vehicles (UAVs), and high altitude platform (HAP) stations  \cite{araniti2021toward,azari2022evolution,giordani2020non,iqbal2023empowering}. These NTNs have the potential of providing broad coverage with ultra-low latency, especially in remote and difficult terrains where terrestrial networks face limitations. The integration of NTNs into the communication infrastructure is considered crucial for extending connectivity to inaccessible areas and bridging the digital divide.

Among the several non-terrestrial options, we are particularly interested in HAPs because of their unique advantages over satellite systems, such as less expensive implementation and deployment due to the absence of space launch costs and relatively simple upgrade, repair, and redeployment \cite{abbasi2024haps}. Positioned in the stratosphere at altitudes between 17 to 22 kilometers above the earth's surface where wind velocities are low and atmospheric turbulence is minimal, HAPs serve as quasi-stationary aerial platforms and have diverse applications in various industries, such as broadcasting, internet connectivity, agriculture, environmental monitoring, emergency communication, surveillance, and disaster monitoring \cite{elamassie2023free}. 

Complementing the capabilities of HAPs, intelligent reflecting surfaces (IRS) emerge as a novel advancement in wireless communication technologies. These systems utilize multiple small, affordable, and energy-efficient devices that can dynamically manipulate wireless signals. IRS improve communication quality, coverage, and security by adjusting the phase and amplitude of reflected signals. They can be utilized on diverse surfaces, such as building facades, ceilings, and advertising boards, seamlessly integrating into existing infrastructures. Furthermore, the use of IRS is crucial in mitigating signal blockages that occur in mm-Wave communication. This leads to enhanced coverage and quality, resulting in more efficient and reliable wireless communication networks \cite{pan2022overview,kumar2022survey,sadia2023irs,okogbaa2022design}.
 While blockage effects are more severe in mm-wave communications, they also pose significant challenges in non-mm-wave scenarios, particularly due to the inherent line-of-sight (LOS) model associated with HAPs. HAPs are designed to provide broad coverage through LOS links and are often proposed for rural, remote, or outdoor environments. However, in urban settings with high buildings or for indoor users, the blockage effect becomes a critical issue, limiting the HAP’s ability to serve all users effectively \cite{HAPsintro1}. 
Our work addresses this problem by proposing the use of IRS between HAP and users as a solution to serve those in non-LOS (NLOS) environments. This approach guides the design of practical communication systems by demonstrating how IRS can be deployed to extend coverage, enhance link reliability, and improve connectivity in environments where direct LOS links are obstructed.

On the other hand, free-space optical (FSO) communication has attracted considerable interest over the years due to its inherent advantages over radio frequency (RF) technology. FSO offers higher data rates, simpler and more cost-effective installation, improved immunity and security, all without the constraints of licensing \cite{sood2018analysis,malik2015free,al2020survey,abd2024performance,rahmani2024massive}. Given these unique characteristics, FSO technology becomes a practical solution for establishing communication channels between ground stations and HAPs.
A thorough analysis of the uplink HAP FSO outage performance was conducted in \cite{SafariHAP2020}, and in order to get minimum error probability, the transmitted laser beam and the receiver's field of view (FoV) were optimized \cite{SafariHAP2020}. Later, a further investigation was carried out in \cite{TsiftsisHAP2022} where the average bit-error-rate (BER) and ergodic capacity were also added, emphasizing the performance optimization through appropriate design of the beam divergence angle, receiver aperture size, and FoV.
While prior research indicates that FSO communication systems present a good alternative to the traditional RF ground-to-HAP links, their overall performance is still limited by issues like turbulence-induced fading, misalignment pointing errors, and sensitivity to weather conditions such as fogs and clouds. Overcoming these challenges is essential to improve the reliability of FSO communication in HAPs-integrated systems. 

In order to achieve end-to-end connectivity characterized by high data rates, flexibility, cost-effectiveness, minimal latency, and energy efficiency, it is very essential to integrate both aerial and terrestrial networks within the communication infrastructure. In this context, this study investigates the end-to-end performance of HAPs-based FSO systems with terrestrial IRS nodes. In such configurations, IRS are deployed to facilitate the transmission of data from the HAP to terrestrial users, when LOS between the HAP and the users is obstructed. Despite the extensive research on HAPs-based FSO links \cite{ata2022haps,deka2022performance,liu2023joint,li2023ris,lou2023haps,10118907}, the integration of IRS has been largely overlooked. 
 By incorporating IRS, our model introduces a three-hop communication system (ground-to-HAP-to-IRS-to-users) instead of the conventional two-hop ground-to-HAP-to-users systems that have been predominantly studied in the literature \cite{10.3389/frcmn.2022.746201,9247498,9320603,7398128,9108615}. 
This novel three-hop architecture significantly differentiates our work by addressing the LOS blockage problem in urban areas and enhancing service to NLOS users, which was not fully explored in previous research.
By installing IRS on tall buildings or other potential obstacles, the signal path can be optimized by controlling and adjusting the reflection angle, allowing HAP signals to bypass obstacles and reach users who are blocked by nearby buildings or trees. In addition, IRS can enhance signal strength through phase adjustment, compensating for signal attenuation during obstacle penetration or long-distance transmission, ensuring that users receive sufficiently strong signals. IRS also actively controls signal reflection paths to reduce unnecessary reflections, thereby minimizing multipath effects and improving communication reliability and stability. Moreover, IRS can dynamically adjust its reflective properties to adapt to user movement and environmental changes, ensuring effective signal coverage to users in various conditions. 
However, the cascaded HAP-IRS-user link involves a more complex distribution compared to the simpler RF models used in previous studies. More specifically, 
the RF communication links between the HAP and IRS, as well as between IRS and users, are modeled using the Shadowed Rician and Nakagami-m fading channels, respectively.
Integrating IRS undoubtedly adds complexity to the RF cascaded link, resulting in a non-centered chi-square distribution. This complexity makes it very challenging to obtain exact closed-form expressions for the end-to-end system performance metrics. Therefore, in our study, we use a mixture Gamma model to accurately approximate the non-centered chi-square distribution. In addition, considering amplify-and-forward (AF) fixed gain at the HAP relay station further complicates the end-to-end analysis due the shift introduced in the expression of the end-to-end signal-to-noise ratio (SNR).

In the ground-to-HAP FSO link, we integrate the impacts of misalignment, modeled by the Hoyt distribution \cite{hpl1}, with the Gamma-Gamma distributed turbulence fading channel \cite{GGmodel}. The Gamma-Gamma distribution is widely recognized as the most appropriate model for describing atmospheric turbulence, particularly in moderate to strong turbulence scenarios. The pointing error, which accounts for deviations in both horizontal and vertical directions, is characterized using the Hoyt distribution \cite{hpl1}.  Most previous studies examining the impact of pointing errors in ground-to-HAP FSO links typically assume that the horizontal and vertical misalignments are independent and identically distributed with Gaussian statistics, sharing the same jitter variance. Consequently, the radial displacement at the receiver is often modeled using a Rayleigh distribution. In contrast, our study employs a generalized pointing error model that removes the assumption of identical jitter variance along both axes, offering a more accurate representation of real-world conditions. Notably, our results can also be simplified to match the commonly adopted scenario where the radial displacement follows a Rayleigh distribution, a widely used model in the literature. Moreover, the FSO link is assumed to operate under two different types of detection methods, namely intensity modulation with direct detection (IM/DD) and heterodyne.

We then conduct a thorough performance analysis of the end-to-end system, wherein the HAP serves as an AF fixed-gain relay terminal between FSO and RF links.
We first derive closed-form analytical expressions for the probability density function (PDF) and cumulative distribution function (CDF) of the end-to-end signal-to-noise ratio (SNR). Using these expressions, we derive analytical formulations for a number of performance metrics, including the ergodic capacity, the average BER for different modulation schemes, and the outage probability (OP). Furthermore, we obtained highly accurate asymptotic results for the OP and the average BER at high SNR conditions using simple functions.
 This work offers, for the first time, a unified analytical framework that can be used with the two types of detection approaches to calculate the fundamental performance metrics of NTNs enhanced with FSO links and IRS. 
The main contributions of this work are summarized as follows:
\begin{itemize}
    \item [--] By incorporating IRS into the RF communication link, our study offers a solution to serve users in NLOS conditions with the HAP.
    \item [--] The integration of IRS introduces complexity to the RF cascaded link, resulting in a non-central chi-square distribution that complicates the derivation of exact closed-form performance metrics. To address this, we use a mixture Gamma model to accurately approximate the distribution of the RF link, characterized by cascaded shadowed-Rician fading for the HAP-to-IRS link and Nakagami-$m$ fading for the IRS-to-users link.
    \item [--] For modeling the ground-to-HAP FSO link, we utilize the Gamma-Gamma distribution, which is considered highly effective for representing atmospheric turbulence under moderate to severe conditions. Moreover, we use the Hoyt distribution to account for pointing errors, as it can effectively capture deviations in both vertical and horizontal directions. Our study also covers both IM/DD and heterodyne detection techniques in FSO, providing a unified mathematical framework that applies to both methods.
    \item [--] We derive closed-form analytical expressions for the PDF, CDF, and moments of the end-to-end SNR. Using these expressions, we formulate analytical expressions for various performance metrics, including ergodic capacity, average BER for different modulation schemes, and OP, all expressed in terms of the bivariate Fox-H function.
     \item [--] We also present accurate asymptotic expressions for the OP and average BER at high SNR. These expressions are formulated using simple functions and are utilized to determine the diversity order of our system.   
\end{itemize}

The subsequent sections of this work are organized as follows. Section II introduces the system and channel models. In Section III, we conduct a statistical analysis of the end-to-end SNR and we derive analytical expressions for various performance metrics, alongside their asymptotic results at high SNR regime. Numerical and simulation results are presented in Section IV, and Section V provides concluding remarks.

\section{CHANNEL AND SYSTEM MODELS}
 Our system model consists of four interconnected nodes: the optical ground station (OGS), HAP, IRS, and users, as shown in Fig.~\ref{fig1}. The communication flow starts with the OGS sending signals to the HAP through a FSO link, leveraging the high bandwidth and large capacity of FSO communication. Acting as an AF relay with a fixed gain, the HAP efficiently receives and amplifies the incoming signals from the OGS before transmitting them to users.
However, practical challenges arise when the direct link between the HAP and users is obstructed by obstacles such as tall buildings or trees, as indicated by the blue dashed line in Fig. 1. These blockages significantly reduce transmission efficiency and quality, making the direct HAP-to-user link unavailable. To overcome this issue, the communication path is redirected through the IRS, strategically positioned between the HAP and users. The IRS effectively controls reflections and directs the signals to the intended users, mitigating the effects of signal degradation and interference commonly seen in traditional RF communication in urban environments, as depicted by the solid blue line in Fig.~\ref{fig1}. RF technology is employed for the HAP-to-user link due to its extensive coverage, which is well-suited for the wide distribution of users. In our work, we assume the HAP operates with a single phased array antenna that serves a single user.
The heights of OGS, HAP, IRS and user are denoted as $H_O$, $H_H$, $H_I$, and $H_U$, respectively. The distances between these entities are represented as $d_{OH}$ (OGS to HAP),  $d_{HI}$ (HAP to IRS), and $d_{IU}$ (IRS to user). The zenith angle of OGS is given by $\zeta$, and the horizontal distances are denoted as $d_{HI0}$ (HAP to IRS) and $d_{IU0}$ (IRS to user).
\begin{figure}[!h]
\centering\includegraphics[scale=0.83, trim={5 10 0 10}, clip]{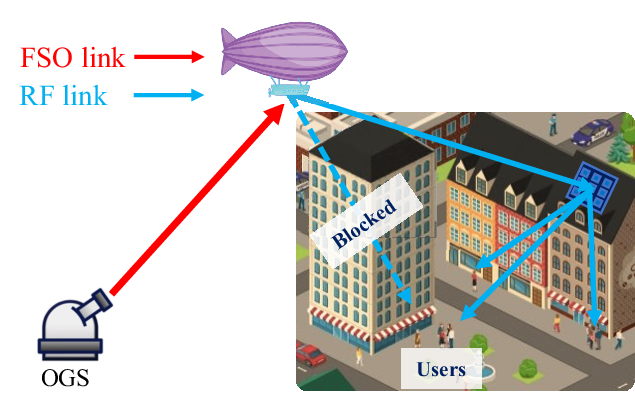}
    \caption{Integrated HAP-ground communication system with IRS.}
    \label{fig1}
\end{figure}
\subsection{FSO Channel}
The FSO channel model for communication from the OGS to the HAP $h$ is accurately characterized by considering the combined effects of atmospheric turbulence $h_{at}$, pointing error $h_{pl}$, and attenuation loss $h_{al}$
\begin{align}
    h=h_{al}h_{at}h_{pl}.
\end{align}
\subsubsection{Attenuation Loss}
The attenuation in the atmosphere, caused by absorption and scattering effects, follows the Beer-Lambert law, as given in \cite[Eq.~(1)]{hal1}
\begin{align}
    h_{al}=\exp\left[-C_h (\lambda)d_{OH}\right],
\end{align}
where $\lambda$ denotes the wavelength in nanometers [nm], $d_{OH}$ is the distance between OGS and HAP, $C_h (\lambda)$ represents the attenuation coefﬁcient which is specified in \cite[Eq.~(4)]{hal1} as
\begin{align}
    C_h(\lambda) = \frac{3.912}{V}\left(\frac{\lambda}{550}\right)^{-q_V},
\end{align}
where  V represents the visibility in kilometers [Km] and the coefficient $q_V$ is provided by \cite[Eq.~(5)]{hal1}
\begin{align}
q_V=
\begin{cases}
1.6, & V>50 \mathrm{Km} \\
1.3, & 6 \mathrm{Km}<V<50 \mathrm{Km} \\
0.585 V^{\frac{1}{3}}, & V<6 \mathrm{Km}
\end{cases}.
\end{align}
\subsubsection{Atmosphere Turbulence}
The Gamma-Gamma distribution is used to represent the PDF of the atmospheric turbulence channel, with its expression provided as \cite[Eq.(56), pp.462]{Laserbeam}
\begin{align}
\label{pdfhat}
    f_{h_{a t}}=\frac{2(\alpha \beta)^{\frac{\alpha+\beta}{2}}}{\Gamma(\alpha) \Gamma(\beta)} h_{a t}^{\frac{\alpha+\beta}{2}-1} K_{\alpha-\beta}\left(2 \sqrt{\alpha \beta h_{a t}}\right), h_{at}>0
\end{align}
where $K_a (\cdot)$ indicates the modified Bessel function of the second kind with the order $a$, and $\Gamma (\cdot)$ represents the Gamma function,
$\alpha=1/[\exp(\sigma_{\ln X}^2)-1]$, and $\beta=1/[\exp(\sigma_{\ln Y}^2 )-1]$. The large-scale log variance $\sigma_{\ln X}^2$ and the small-scale log variance $\sigma_{\ln Y}^2$ are provided by \cite[Eqs. (97) and (101), pp.352]{Laserbeam} as
\begin{align}
    \sigma_{\ln X}^2=\frac{0.49\sigma_B^2}{[1+0.56(1+\Theta) \sigma_B^{12/5} ]^{7/6}},
\end{align}
\begin{align}
    \sigma_{\ln Y}^2=\frac{0.51\sigma_B^2}{[1+0.69\sigma_B^{12/5} ]^{5/6}},
\end{align} 
where $\sigma_B^2$ represents the Rytov variance, which quantifies the scintillation index in weak turbulence scenarios. For the uplink propagation case, $\sigma_B^2$ is expressed as \cite[Eq.~(12)]{ata2022haps}
\begin{align}
       \nonumber  \sigma_{B}^{2} =  \,\,&8.7 {k_w}^{7 / 6}\left(H_H-H_O\right)^{5 / 6} \sec ^{11 / 6}(\zeta) \\ 
       \nonumber &\operatorname{Re}\Big(\int_{H_O}^{H_H} C_{n}^{2}(l)\Big\{\left[\Lambda {\xi_0}^{2}+\mathrm{i} \xi_0(1-\bar{\Theta} \xi_0)\right]^{5 / 6}\\ 
       &-\Lambda^{5 / 6} {\xi_0}^{5 / 3}\Big\} \mathrm{d} l\Big),
\end{align}
where $k_w=2\pi/\lambda$ is the wave number with $\lambda$ being the wavelength in meters [m], $\sec(\cdot)$ represents the Secant funtion, $\zeta$ denotes the zenith angle, $\Lambda=\Lambda_0/(\Lambda_0^2+\Theta_0^2 ) $ stands for the Fresnel ratio of the Gaussian beam at the receiver, $\Lambda_0=2d_{OH}/(k_wW_0^2 )$ where $W_0$ is the beam radius, $\Theta_0=1-d_{OH}/F_0$ represents the beam curvature parameter at the transmitter, ${\xi_0}=(l-H_H)/(H_O-H_H)$ is the normalized distance parameter for the uplink propagation case, $\bar{\Theta}=1-\Theta$ is the complementary parameter, and $\Theta=\Theta_0/(\Theta_0^2+\Lambda_0^2 )$ denotes the beam curvature parameter at the receiver. $C_n^2 (l)$ stands for the turbulence structure constant according to the conventionally used Hufnagel-Valley (HV) model, which is given by \cite[Eq.~(1), pp.481]{Laserbeam} 
\begin{align}
   \nonumber  C_{n}^{2}(l)=&0.00594(\omega / 27)^{2}\left(10^{-5} l\right)^{10}\exp (-l / 1000)\\ 
   \nonumber &+2.7 \times 10^{-16} \exp (-l / 1500)\\ 
   &+A \exp (-l / 1000),
\end{align}
where $l$ is in meters [m], $\omega$ is the root mean square (rms) windspeed in meters per second [m/s] and $A$ denotes the nominal value of $C_{n}^{2}(0)$, as detailed in \cite[pp.481]{Laserbeam}.
\subsubsection{Pointing Error}
In this paper, we employ a generalized model for the pointing error, which considers distinct jitter values for the horizontal and vertical displacements of the beam, as described by the Hoyt model.
The  PDF of $h_{pl}$ under these conditions is given by \cite[Eq.~(11)]{hpl1}
\begin{align}
\label{pdfhpl}
    f_{h_{p l}}\left(h_{p l}\right)=\frac{\eta_{s}^{2}}{2 \pi q_{H}} \int_{-\pi}^{\pi} \frac{h_{p l}^{\eta_{s}^{2} \xi(\varphi)-1}}{A_{0}^{\eta_{s}^{2} \xi(\varphi)}} d \varphi,
\end{align}
where $\eta_{s}$ is defined as $\omega_{b}\sqrt{\sqrt{\pi A_0}  /\left[2 v_e\exp \left(-v_e^{2}\right)\right]} /\left(2 \sigma_{s}\right)$,$v_e=r_{a}\sqrt{\pi / 2}  / \omega_{b}$ with $\omega_b$ representing the beam waist, $A_0=\text{erf}^2(v_e)$, and ${\rm erf}(\cdot)$ denoting the error function. The parameter $r_a$ stands for the radius of the receiver aperture. In addition, $q_H = \sigma_z /\sigma_s$, where $\sigma_s$ and $\sigma_z$ represent the variances of beam jitters in the vertical and horizontal directions, respectively and $\xi(\varphi)=[1-\left(1-q_{H}^{2}\right) \cos ^{2} \varphi]/q_{H}^{2}$.

The PDF of the combined channel state $h$ can be written as
\begin{align}
\label{PDFh0}
    f_{h}(h)=\int_{\frac{h}{A_{0} h_{a l}}}^{\infty} f_{h_{p l}}\left(\frac{h}{h_{a t} h_{a l}}\right) \frac{f_{h_{a t}}\left(h_{a t}\right)}{h_{a t} h_{a l}}  dh_{a t}.
\end{align}
Substituting (\ref{pdfhat}) and (\ref{pdfhpl}) into (\ref{PDFh0}) yields
\begin{align}
    \nonumber  f_{h}(h)&=\int_{\frac{h}{A_{0} h_{a l}}}^{\infty} \frac{\eta_{s}^{2}}{2 \pi q_{H}} \int_{-\pi}^{\pi} \frac{\left(\frac{h}{h_{a t} h_{a l}}\right)^{\eta_{s}^{2} \xi(\varphi)-1}}{A_{0}^{\eta_{s}^{2} \xi(\varphi)}} d\varphi \\&
     \times\frac{\frac{2(\alpha \beta)^{(\alpha+\beta) / 2}}{\Gamma(\alpha) \Gamma(\beta)} h_{a t}^{\frac{(\alpha+\beta)}{2}-1} K_{\alpha-\beta}\left(2 \sqrt{\alpha \beta h_{a t}}\right)}{h_{a t} h_{a l}} dh_{a t}.
\end{align}
Employing \cite[Eq.~(14)]{MeijerGalgorithm} to represent $K_{\alpha-\beta}(\cdot)$ using the Meijer-G function along with \cite[Eq.~(07.34.21.0085.01)]{Wolfram} and \cite[Eq.~(1.60)]{Htran2}, we can obtain the PDF of $h$ as
\begin{align}
 \label{PDFh}
\nonumber f_{h}\left(h\right)&=\frac{\eta_{s}^{2} }{2\pi q_H \Gamma(\alpha) \Gamma(\beta) h}  \\ 
&\times \int_{-\pi}^{\pi}{\rm G}_{1,3}^{3,0}\left[\frac{\alpha \beta h}{A_{0} h_{a l}  } \Bigg| \begin{array}{c}
1+\eta_{s}^{2}\xi(\varphi) \\
\eta_{s}^{2}\xi(\varphi), \alpha, \beta
\end{array}\right]d\varphi,
\end{align}
where ${\rm G}^{m,n}_{p,q}[\cdot]$ represents the Meijer-G function \cite[Eq.~(9.301)]{intetable}.

The received signal at the HAP from the OGS can be expressed as
\begin{align}
    y_{H}(t)=\eta P_{0} h s(t)+n_{H}(t),
\end{align}
where $P_0$ denotes the transmit power at the OGS, $\eta$ represents the optical-to-electrical conversion coefficient, and $n_{H}(t) \sim \mathcal{N}\left(0, \sigma_{H}^{2}\right)$ refers to the Gaussian noise process with zero mean value and variance $\sigma_{H}^{2}$. Then the SNR of the FSO link $\gamma_{H}$ can be formulated considering both IM/DD and heterodyne detections as
\begin{align}
\label{gammaHbar}
    \gamma_{H}=\overline{\gamma_{H}} h^{r},
\end{align}
where $\overline{\gamma_{H}}=\frac{\left(\eta P_{0}\right)^{r}}{\sigma_{H}^{2}}$ and the parameter $r$ varies depending on the type of detection technique employed, with $r=1$ for the heterodyne technique and $r=2$ for IM/DD.

The PDF of $\gamma_H$ can be obtained from (\ref{PDFh}) by applying the random variable transformation in (\ref{gammaHbar}) as
\begin{align}
 \label{PDFH}
\nonumber &f_{\gamma_{H}}\left(\gamma_{H}\right)=\frac{\eta_{s}^{2} }{2\pi q_Hr \Gamma(\alpha) \Gamma(\beta)\gamma_{H}} \\ 
& \times \int_{-\pi}^{\pi}{\rm G}_{1,3}^{3,0}\left[\frac{\alpha \beta}{A_{0} h_{a l}  }\left(\frac{ \gamma_{H}}{\overline{\gamma_{H}}}\right)^{\frac{1}{r}} \Bigg| \begin{array}{c}
1+\eta_{s}^{2}\xi(\varphi) \\
\eta_{s}^{2}\xi(\varphi), \alpha, \beta
\end{array}\right]d\varphi.
\end{align}
Then, by using (\ref{PDFH}) and applying \cite[Eq.~(2.24.2.3)]{prudnikov1}, we get the CDF of $\gamma_H$ as
\begin{align}
\label{CDFH}
\nonumber &F_{\gamma_{H}}\left(\gamma_{H}\right) =1-\frac{\eta_{s}^{2} }{2\pi q_H \Gamma(\alpha) \Gamma(\beta)}\\ 
&\times \int_{-\pi}^{\pi}{\rm G}_{2,4}^{4,0}\left[\frac{\alpha \beta}{A_{0} h_{a l} }\left(\frac{ \gamma_{H}}{\overline{\gamma_{H}}}\right)^{\frac{1}{r}} \Bigg| \begin{array}{c}
1+\eta_{s}^{2}\xi(\varphi), 1 \\
0, \eta_{s}^{2}\xi(\varphi), \alpha, \beta
\end{array}\right]d\varphi.
\end{align}
\subsection{RF Channel}
\subsubsection{Shadowed Rician Fading Channel}
The channel state from the HAP to the IRS, denoted by $\alpha_{si}$, is characterized as a shadowed Rician (SR) channel, with the PDF of $\left|\alpha_{si}\right|^{2}$ described in \cite[Eq.~(6)]{showedrician} as
\begin{align}
\nonumber & f_{\left|\alpha_{si}\right|^{2}}(x)=\left[\frac{2 b_R m_R}{2 b_R m_R+\Omega_R}\right]^{m_R}\frac{\exp{\left(-\frac{x}{2b_R}\right)}}{2 b_R} \\ 
& \times { }_{1} F_{1}\left(m_R , 1 , \frac{\Omega_R\, x}{2 b_R(2 b_R m_R+\Omega_R)} \right), \quad x\ge 0
\end{align}
where ${ }_{1} F_{1}(\cdot)$ stands for the confluent hypergeometric function, $m_R$ denotes the fading severity parameter, $2b_R$ is the average power of the multipath component, and $\Omega_R$ represents the average power of the line-of-sight (LOS) component \cite{showedrician}.

The $s$-th moments of the SR fading channel are given in terms of the Gauss hypergeometric function as \cite[Eq.~(5)]{showedrician}
\begin{align}
\nonumber &\mathbb{E}\left({\alpha_{si}}^{s}\right)=\left(\frac{2 b_{R} m_{R}}{2 b_{R} m_{R}+\Omega_{R}}\right)^{m_{R}}\left(2 b_{R}\right)^{\frac{s}{2}} \\ 
&\times \Gamma\left(\frac{s}{2}+1\right){ }_{2} F_{1}\left(\frac{s}{2}+1, m_{R}, 1, \frac{\Omega_{R}}{2 b_{R} m_{R}+\Omega_{R}}\right).
\end{align}
\subsubsection{Nakagami-m Fading Channel}
The channel state from the IRS to users, represented as $\beta_{si}$, is assumed to follow the Nakagami-m fading channel with its PDF given as \cite[Eq.~(2.3/67)]{digital}
\begin{align}
    f_{\beta_{si}}(x)={\left(\frac{m_{N}}{\sigma_{N}^{2}}\right)}^{m_N} \frac{2 x^{2 m_{N}-1}}{\Gamma\left(m_{N}\right)} \exp \left(-\frac{m_{N} x^{2}}{\sigma_{N}^{2}}\right), \quad x\ge 0
\end{align}
where $m_N$ represents the fading parameter.
The $s$-th moments of $\beta_{si}$ are provided in \cite{digital} as 
\begin{align}
\mathbb{E}\left({\beta_{si}}^{s}\right)=\frac{\Gamma\left(m_{N}+\frac{s}{2}\right)}{\Gamma\left(m_{N}\right)}\left(\frac{\sigma_{N}^{2}}{m_{N}}\right)^{\frac{s}{2}}.
\end{align}
\subsubsection{Pathloss Modeling} 
The pathloss from the HAP to the IRS, indicated by $P L_{(\mathrm{HI})}$, can be expressed as
\begin{align}
    P L_{(\mathrm{HI})}[\mathrm{dB}]=L_{\mathrm{FSPL}}+L_{\mathrm{Rain}}+L_{\mathrm{Atm} }+L_{\text {oth }},
\end{align}
where $L_{\mathrm{Rain}}$ denotes the rain-induced loss in decibels per kilometer (dB/km), $L_{\mathrm{Atm} }$ represents the gaseous atmospheric absorption loss, and $L_{\text {oth}}$ stands for miscellaneous losses (dB). Further, $L_{\mathrm{FSPL}}$ is the free-space 
path loss (dB), given as $L_{\mathrm{FSPL}} = 92.45 + 20 \log _{10}(f_c)+20\log_{10}(d_{HI})$ with $f_c$ being the carrier 
frequency (in GHz), and $d_{HI}$ is the distance between the HAP and user nodes.
In addition, the pathloss from the IRS to the end users, $P L_{(\mathrm{IU})}$, can be formulated as
\begin{align}
   \nonumber P L_{\left(\mathrm{IU}\right)}[\mathrm{dB}]&=40 \log _{10}\left(d_{\mathrm{I} \mathrm{U}}\right)-20 \log _{10}\left(H_{I}\right)\\ 
   &-20 \log _{10}\left(H_{U}\right),
\end{align}
where $d_{\mathrm{I} \mathrm{U}}$ represents the distance between the IRS and users, and  $H_I$ stands for the height of building where the IRS is placed.

The received signal at the the user node can be given as
\begin{align}
    y_{U}(t)=\sqrt{P_{h}} \boldsymbol{g}_{s}^{T} \zeta_{I} \boldsymbol{h}_{s} s(t)+n_{U}(t),
\end{align}
where $\boldsymbol{h_s}$ and $\boldsymbol{g_s}$ are the channel vectors for HAP to IRS and IRS to users, respectively, given as $\boldsymbol {h_s}=[h_{s1},\ldots,h_{sN} ]^T$ and $\boldsymbol {g_s}=[g_{s1},\ldots,g_{sN }]^T$, with $h_{si}=\alpha_{si} \exp{({\mathrm{i}\theta_{hi} })}$ and  $g_{si}=\beta_{si} \exp{({\mathrm{i}\theta_{gi}}) }$. $\zeta_{I}=\operatorname{diag}\left[\varrho_{1} \exp{({-\mathrm{i} \phi_{I 1}})}, \ldots \varrho_{N} \exp{({-\mathrm{i}\phi_{I N}})}\right]$ is a matrix of IRS meta-surface induced complex valued reflection coefficient with attenuation coefficient $\varrho \in[0,1] $ and phase shift $\phi_{I} \in[0,2 \pi]$, and $n_{U}(t) \sim \mathcal{N}\left(0, {\sigma_{U}}^{2}\right)$. In addition, $P_h$ can be expressed as
\begin{align}
     \nonumber P_h[dB]& = P_{0}[dB]-PL_{(\mathrm{HI})}[dB]\\&-PL_{(\mathrm{IU})}[dB]+G_{Tx}[dB]+G_{Rx}[dB],
\end{align}
where $G_{\mathrm{Tx}}$ refers to the transmitter antenna gain (dB) and $G_{\mathrm{Rx}}$ represents the receiver antenna gain (dB).

Finally, the SNR for the RF channel, $\gamma_{U}$, can be formulated as
\begin{align}
\gamma_{U}=\overline{\gamma_{U}}\left|\sum_{i=1}^{N} \alpha_{si} \beta_{si} \varrho_{si} \exp\left[{-\mathrm{i}\left(\Phi_{i}-\theta_{hsi}-\theta_{gsi}\right)}\right]\right|^{2},
\end{align}
where $\overline{\gamma_U}={P_{h}}/{{\sigma_U}^{2}}$. 
 In our analysis, we assume that the channel phases of $h_{si}$ and $g_{si}$ for $i=1,\ldots,N$ are perfectly known at the RIS. This assumption represents the optimal scenario for system performance and serves as a benchmark for evaluating practical applications \cite{IRSBasar}.
To achieve the highest possible SNR at the user node, the reflection coefficient induced by the IRS is selected in an optimal manner such that $\forall i \,\varrho_{si}=1 $. Moreover, the optimal choice for 
$\Phi_{si}$ which maximizes the instantaneous SNR is given by $\Phi_{si}=\theta_{hsi}+\phi_{gi}$ for $i=1,\ldots,N$ \cite{IRSBasar,RFchannel}. \cite{Badiu20} investigates how errors in reflector phases impact error probability and demonstrates that the system remains robust, with communication performance proving notably resilient even with imperfect phase estimation.
Therefore, the optimal maximized SNR for the RF channel is expressed as
\begin{equation}
\gamma_{U}=\overline{\gamma_{U}}\left|\sum_{i=1}^{N} \alpha_{si} \beta_{si}\right|^{2}.
\end{equation}
\begin{lemma}
 The PDF of the SNR for the RF link from the HAP to the user via the IRS, $\gamma_U$, can be expressed as
\begin{align}
\label{PDFU}
\nonumber&f_{\gamma_{U}}\left(\gamma_{U}\right)=\frac{\left(\overline{\gamma_{U}} {\mu_{Z}}^{2}\right)^{\frac{1}{4}}}{2 \overline{\gamma_{U}} {\sigma_{Z}}^{2}} \exp \left(-\frac{{\mu_{Z}}^{2}}{2  {\sigma_{Z}}^{2}}\right)\\
 &\times{\gamma_{U}}^{-\frac{1}{4}} \exp \left(-\frac{\gamma_{U}}{2 \overline{\gamma_{U}} {\sigma_{Z}}^{2}}\right)  I_{-\frac{1}{2}}\left(\frac{\mu_{Z}}{{\sigma_{Z}}^{2}} \sqrt{\frac{\gamma_{U}}{\overline{\gamma_{U}}}}\right),
\end{align}
where $\mu_Z=N\mathbb{E}({\alpha_{si}})\mathbb{E}({\beta_{si}})$, ${\sigma_Z}^2=N[\mathbb{E}({\alpha_{si}}^2)\mathbb{E}({\beta_{si}}^2)-{(\mu_Z/N)}^2]$, and $I_v(\cdot)$ represents the $v$th-order modified Bessel function of the first kind. This expression is particularly accurate when the number of IRS units, 
$N$, is large, with its accuracy improving as 
$N$ increases. This highlights the impact of IRS size on the performance of the RF link.
\end{lemma}
\begin{proof}
See Appendix \ref{A}.
\end{proof}
\begin{figure*}[b]
\hrule
\begin{align*}
\label{finalCDFA}
\nonumber F_{\gamma}(\gamma)\underset{\overline{\gamma_H}\gg 1}{\mathop{\approx }}
& -\frac{ \eta_{s}^{2}}{2 \pi q_{H} \Gamma(\alpha) \Gamma(\beta)} \int_{-\pi}^{\pi} \sum_{i=1}^{N_{x}} \alpha_{x i}\\
\nonumber &
\times \Bigg\{\frac{C^{\beta_{x i}} \Gamma\left(\alpha-\eta_{s}^{2} \xi(\varphi)\right) \Gamma\left(\beta-\eta_{s}^{2} \xi(\varphi)\right)}{r \Gamma\left(1-\frac{\eta_{s}^{2} \xi(\varphi)}{r}\right)}\left[\frac{\alpha \beta}{A_{0} h_{a l}}\left(\frac{\gamma}{\overline{\gamma_{H}}}\right)^{\frac{1}{r}}\right]^{\eta_{s}^{2} \xi(\varphi)} {\rm G}_{1,2}^{2,1}\left[C \zeta_{x i}  \begin{array}{|c}
1-\beta_{x i}+\frac{\eta_{s}^{2} \xi(\varphi)}{r} \\
0,-\beta_{x i}
\end{array}\right]\\
\nonumber &+\frac{C^{\beta_{x i}} \Gamma(\beta-\alpha)}{r \Gamma\left(1-\frac{\alpha}{r}\right)\left(\eta_{s}^{2} \xi(\varphi)-\alpha\right)}\left[\frac{\alpha \beta}{A_{0} h_{a l} }\left(\frac{\gamma}{\overline{\gamma_{H}}}\right)^{\frac{1}{r}}\right]^{\alpha} {\rm G}_{1,2}^{2,1}\left[C \zeta_{x i} \begin{array}{|c}
1-\beta_{x i}+\frac{\alpha}{r} \\
0,-\beta_{x i}
\end{array}\right] \\&+\frac{C^{\beta_{x i}} \Gamma(\alpha-\beta)}{r \Gamma\left(1-\frac{\beta}{r}\right)\left(\eta_{s}^{2} \xi(\varphi)-\beta\right)}\left[\frac{\alpha \beta}{A_{0} h_{a l}}\left(\frac{\gamma}{\overline{\gamma_{H}}}\right)^{\frac{1}{r}}\right]^{\beta} {\rm G}_{1,2}^{2,1}\left[C \zeta_{x i} \begin{array}{|c}
1-\beta_{x i}+\frac{\beta}{r} \\
0,-\beta_{x i}
\end{array}\right]\Bigg\} d \varphi.\tag{36}
\end{align*}
\hrule
\end{figure*}
\noindent The CDF of $\gamma_U$ is given as \cite[Eq.~(2.3-35)]{digital}
\begin{align}
\label{CFU}
F_{\gamma_{U}}\left(\gamma_{U}\right)=1-Q_{\frac{1}{2}}\left(\frac{\mu_{Z}}{\sigma_{Z}}, \frac{\sqrt{\gamma_{U}}}{\sqrt{\overline{\gamma_{U}}} \sigma_{Z}}\right),
\end{align}
where $Q_m(a,b)$ denotes the generalized Marcum Q-function.

Given the complex form and mathematical intractability of the non-centred chi-square distribution, we employ the mixture Gamma model \cite{mixturegamma} to approximate it, as demonstrated in Lemma 2.
\begin{lemma}
The PDF of $\gamma_{U}$ can be represented using a mixture Gamma distribution as
    \begin{align}
    \label{PDFUa}
        f_{\gamma_{U}}(\gamma_{U}) \cong \sum_{i=1}^{N_{x}} \alpha_{xi} {\gamma_{U}}^{\beta_{xi}-1} \exp \left(-\zeta_{xi} \gamma_{U}\right),
    \end{align}
with parameters obtained by using PDF matching as
    \begin{align}  \label{PDFparam}
        \begin{cases}
        \alpha_{x i}=\frac{\theta_{xi}}{\sum_{j=1}^{N_{i}} \theta_{xj} \Gamma\left(\beta_{x j}\right) \zeta_{x j}^{-\beta_{x j}}}, 
        \\[2ex]
        \beta_{x i}=-\frac{1}{2}+i,
 \\[2ex]
 \zeta_{x i}=({2 \overline{\gamma_{U}} {\sigma_{Z}}^{2}})^{-1},
 \\[2ex]
 \theta_{x i}=\frac{\left(\overline{\gamma_{U}} {\mu_{Z}}^{2}\right)^{\frac{1}{4}}}{2 \overline{\gamma_{U}} {\sigma_{Z}}^{2}}\frac{ \exp \left(-\frac{{\mu_{Z}}^{2}}{2 {\sigma_{Z}}^{2}}\right)}{(i-1) ! \Gamma\left(i-\frac{1}{2}\right)}\left(\frac{\mu_{Z}}{2{\sigma_{Z}}^{2}\sqrt{{\overline{\gamma_{U}}}}} \right)^{-\frac{5}{2}+2  i}.
\end{cases}
    \end{align}
    \begin{proof}
See Appendix \ref{B}.
\end{proof}
\end{lemma}
 It is important to mention that (\ref{PDFUa}) effectively simplifies the complex non-central chi-square distribution of the HAP-to-IRS-to-users RF link and shows a strong alignment with Monte-Carlo simulation results. Moreover, the accuracy of this expression improves as the number of Gamma distributions, $N_x$, increases.
\section{END-TO-END SYSTEM PERFORMANCE}
\subsection{End-to-End SNR Statistics}
The end-to-end SNR in the fixed-gain relaying scheme can be obtained by considering the expression given by \cite[Eq.~(28)]{fixedgain}, assuming that saturation can be ignored as
\begin{align}\label{TSNR}
\gamma=\frac{\gamma_{H} \gamma_{U}}{\gamma_{U}+C},
\end{align}
where $C$ represents a constant relay gain.
\begin{lemma}
The CDF and PDF of the overall SNR can be derived in terms of the bivariate Fox-H function, also referred to as the Fox-H function of two variables whose implementation is provided in \cite{Peppas} as 
\begin{align}
    \label{finalCDF}
\nonumber &F_{\gamma}(\gamma)={1-\frac{ \eta_{s}^{2} }{2\pi q_Hr \Gamma(\alpha) \Gamma(\beta)}   \sum_{i=1}^{N_{x}}\alpha_{xi} \zeta_{xi}^{-\beta_{xi}}\int_{-\pi}^{\pi}{\rm H}_{1,0 ; 0,2  ; 3,2}^{0,1;2,0 ; 0,3}}\\
&\times\footnotesize{\left[\!\!\left.\begin{matrix}
\left(1 ; 1, \frac{1}{r}\right) \\
-\\ -\\
(0,1)\left(\beta_{xi}, 1\right) \\
\left(1-\eta_{s}^{2}\xi(\varphi), 1\right)(1-\alpha, 1)(1-\beta, 1) \\
\left(-\eta_{s}^{2}\xi(\varphi), 1\right)\left(0, \frac{1}{r}\right)
\end{matrix}\right|{\zeta_{xi} C, \frac{A_{0} h_{a l}\left(\frac{\overline{\gamma_{H}}}{\gamma}\right)^{\frac{1}{r}} }{\alpha \beta}}\!\!\right]}\!\!\!d\varphi,
    \end{align}
and
\begin{align}
 \label{finalPDF}
\nonumber &f_{\gamma}(\gamma)=\frac{ \eta_{s}^{2}}{2 \pi q_{H}r \Gamma(\alpha) \Gamma(\beta)\gamma}   \sum_{i=1}^{N_{x}} \alpha_{xi} \zeta_{xi}^{-\beta_{xi}} \int_{-\pi}^{\pi}{\rm H}_{1,0 ; 0,2  ; 3,2}^{0,1;2,0 ; 0,3}\\
&\times \footnotesize{\left[\!\!\left.\begin{matrix}
\left(1 ; 1, \frac{1}{r}\right) \\
-\\ -\\
(0,1)\left(\beta_{xi}, 1\right) \\
\left(1-\eta_{s}^{2}\xi(\varphi), 1\right)(1-\alpha, 1)(1-\beta, 1) \\
\left(-\eta_{s}^{2}\xi(\varphi), 1\right)\left(1, \frac{1}{r}\right)
\end{matrix} \right|{ \zeta_{xi} C, \frac{A_{0} h_{a l} \left(\frac{\overline{\gamma_{H}}}{\gamma}\right)^{\frac{1}{r}}}{\alpha \beta}}\!\!\right]}d\varphi.
\end{align}
\begin{proof}
See Appendix \ref{C}.
\end{proof}
\end{lemma} 
\begin{lemma}
The $s$-th moments of $\gamma$ can be demonstrated as
\begin{align}
\label{MomentsH}
\nonumber \mathbb{E}\left(\gamma^{s}\right)&=\frac{ \eta_{s}^{2} }{2 \pi q_{H} \Gamma(\alpha) \Gamma(\beta)}\left[\left(\frac{A_{0} h_{a l} }{\alpha \beta}\right)^{r} \overline{\gamma_{H}}\right]^{s} \\
\nonumber & \times \int_{-\pi}^{\pi} \sum_{i=1}^{N_{x}} \alpha_{x i} {\zeta_{x i}}^{-\beta_{x i}} {\rm G}_{1,2}^{2,1}\left[\zeta_{x i} C\bigg| \begin{matrix} 1-s \\0, \beta_{x i} \end{matrix}\right] \\
& \times \frac{\Gamma\left(\eta_{s}^{2} \xi(\varphi)+s r\right) \Gamma(\alpha+s r) \Gamma(\beta+s r)}{\Gamma\left(1+\eta_{s}^{2} \xi(\varphi)+s r\right) \Gamma(s)} d\varphi.
\end{align}
\begin{proof}
See Appendix \ref{D}.
\end{proof}
\end{lemma}
\begin{table*}[b]
\hrule
\vspace{1.5ex}
    \caption{MODULATION PARAMETERS}
    \centering
    \begin{tabular}{|l|l|l|l|l|l|}
\hline \text { Modulation } & $\boldsymbol{\delta_B}$ & $\boldsymbol{p_B}$ & $\boldsymbol{q}_{\boldsymbol{Bk}}$ & $\boldsymbol{N_B}$ & \text { Detection } \\
\hline \text { M-PSK } & $\frac{2}{\max \left(\log _{2} M, 2\right)}$ & 1 / 2 & $\sin ^{2}\left(\frac{(2 k-1) \pi}{M}\right)\log _{2} M$ & $\max \left(\frac{M}{4}, 1\right)$ & \text { Heterodyne } \\
\hline \text { M-QAM } & $\frac{4}{\log _{2} M}\left(1-\frac{1}{\sqrt{M}}\right)$ & 1 / 2 &$\frac{3(2 k-1)^{2}}{2(M-1)}\log _{2} M $& $\frac{\sqrt{M}}{2}$ & \text { Heterodyne } \\
\hline \text { OOK } & 1 & 1 / 2 & 1 / 2 & 1 & \text { IM/DD } \\
\hline
\end{tabular}
    \label{tab:my_label}
\end{table*}
\begin{figure*}[b]
\hrule
\begin{align*}\label{ABEREA}
\nonumber &I\left(p_B, q_{B k}\right)  \underset{\overline{\gamma_H}\gg 1}{\mathop{\approx }}-\frac{\eta_s^2}{4 \pi q_H \Gamma(\alpha) \Gamma(\beta)\Gamma\left(p_B\right)}   \int_{-\pi}^\pi \sum_{i=1}^{N_x} \alpha_{x i} \Bigg\{{(q_{Bk})}^{\frac{-\eta_s^2 \xi(\varphi)}{r}} \Gamma\left(\frac{\eta_s^2 \xi(\varphi)}{r}+p_B\right)  \\
\nonumber & \times \frac{C^{\beta_{x i}} \Gamma\left(\alpha-\eta_s^2 \xi(\varphi)\right) \Gamma\left(\beta-\eta_s^2 \xi(\varphi)\right)}{r \Gamma\left(1-\frac{\eta_s^2 \xi(\varphi)}{r}\right)}\left[\frac{\alpha \beta}{A_0 h_{a l}}\left(\frac{1}{\overline{\gamma_H}}\right)^{\frac{1}{r}}\right]^{\eta_s^2 \xi(\varphi)} {\rm G}_{1,2}^{2,1}\left[\begin{array}{c|c}
C \zeta_{x i} & \begin{array}{c}1-\beta_{x i}+\frac{\eta_s^2 \xi(\varphi)}{r} \\ 
0,-\beta_{x i}\end{array}
\end{array}\right] \\
\nonumber & +{(q_{B k})}^{-\frac{\alpha}{r} }\Gamma\left(\frac{\alpha}{r}+p_B\right) \frac{C^{\beta x i} \Gamma(\beta-\alpha)}{r \Gamma\left(1-\frac{\alpha}{r}\right)\left(\eta_s^2 \xi(\varphi)-\alpha\right)}\left[\frac{\alpha \beta}{A_0 h_{a l}}\left(\frac{1}{\overline{\gamma_H}}\right)^{\frac{1}{r}}\right]^\alpha {\rm G}_{1,2}^{2,1}\left[\begin{array}{c|c}
C \zeta_{x i} & \begin{array}{c}
1-\beta_{x i}+\frac{\beta}{r} \\
0,-\beta_{x i}
\end{array}
\end{array}\right] \\
& +{(q_{B k})}^{-\frac{\beta}{r}} \Gamma\left(\frac{\beta}{r}+p_B\right) \frac{C^{\beta x i} \Gamma(\alpha-\beta)}{r \Gamma\left(1-\frac{\beta}{r}\right)\left(\eta_s^2 \xi(\varphi)-\alpha\right)}\left[\frac{\alpha \beta}{A_0 h_{a l}}\left(\frac{1}{\overline{\gamma_H}}\right)^{\frac{1}{r}}\right]^\beta {\rm G}_{1,2}^{2,1}\left[\begin{array}{c|c}
C \zeta_{x i} & \begin{array}{c}
1-\beta_{x i}+\frac{\beta}{r} \\
0,-\beta_{x i}
\end{array}
\end{array}\right]\Bigg\} d\varphi.\tag{42}
\end{align*}
\hrule
\end{figure*}
\subsection{Performance Analysis}
\subsubsection{Outage Probability}
The OP refers to the probability that the end-to-end SNR is lower than a predetermined threshold $\gamma_{\rm{th}}$. By replacing $\gamma$ with $\gamma_{\rm{th}}$ in (\ref{finalCDF}), a unified expression for the OP in both detection methods can be easily derived.

 In (\ref{finalCDF}), the CDF is expressed using the bivariate Fox-H function, which is complex and not commonly available in
widely used mathematical software like MATLAB or MATHEMATICA. To address this and gain useful insights, we analyze the CDF in the high SNR regime using an asymptotic approach. This leads to a simplified expression for the CDF in (\ref{finalCDF}) that only involves elementary and standard functions already built into MATLAB and MATHEMATICA as shown in (\ref{finalCDFA}). 
\begin{proof}
See Appendix \ref{E}.
\end{proof}
 Besides its simplicity, this approximation in (\ref{finalCDFA}) is highly accurate and converges well to the exact result at high SNR levels. Moreover, this asymptotic result is especially useful for calculating the system's diversity order. Specifically, when $q_H=1$, the diversity gain of our system can be computed as follows:
\addtocounter{equation}{1}
\begin{align}
\label{diversity}
\mathcal{G}_{d}=\min \left(\frac{\alpha}{r}, \frac{\beta }{r}, \frac{\eta_{s}^{2} }{r}\right).
\end{align}

\subsubsection{Average Bit-Error Rate}
A compact and unified expression of the average BER for various coherent M-QAM and M-PSK modulation schemes, as well as IM/DD OOK modulation technique can be provided as \cite[Eq.~(22)]{dualhopFSO}
\begin{align}
\overline{P_{e}}=&\frac{\delta_{B}q_{Bk}^{p_{B}}}{2 \Gamma\left(p_{B}\right)} \sum_{k=1}^{N_{B}} \int_{0}^{\infty} \gamma^{p_{B}-1} \exp \left(-q_{B k} \gamma\right) F_{\gamma}(\gamma) d \gamma,
\end{align}
where $N_{B}$, $\delta_{B}$ , $p_{B}$, and $q_{B k}$ are detailed in Table \ref{tab:my_label}.
\begin{lemma}
Define  $I\left(p_{B}, q_{B k}\right)$ as
\begin{equation}
\label{ABERI}
\begin{aligned}
    I\left(p_{B}, q_{B k}\right)=&\frac{q_{Bk}^{p_{B}}}{2 \Gamma\left(p_{B}\right)} \int_{0}^{\infty} \gamma^{p_{B}-1} \exp \left(-q_{B k} \gamma\right) F_{\gamma}(\gamma) d \gamma,
\end{aligned}
\end{equation}
which can be evaluated as
\begin{align}
\nonumber &I\left(p_{B}, q_{B k}\right)=\frac{1}{2}-\frac{ \eta_{s}^{2} }{4 \pi rq_{H} \Gamma\left(p_{B}\right) \Gamma(\alpha) \Gamma(\beta)} 
\\ \nonumber &\times \int_{-\pi}^{\pi} \sum_{i=1}^{N_{x}} \alpha_{x i} {\zeta_{x i}^{-\beta_{x i}}} {\rm H}_{1,0 ; 0, 2 ; 3,3}^{0,1 ; 2,0 ; 1,3}\\&\footnotesize\left[\!\!\!\!\!\begin{array}{c|c}{\begin{array}{c}
\left(1 ; 1, \frac{1}{r}\right) \\
-\\- \\
(0,1)\left(\beta_{x i}, 1\right) \\
\left(1-\eta_{s}^{2} \xi(\varphi), 1\right)(1-\alpha, 1)(1-\beta, 1) \\
\left(p_{B}, \frac{1}{r}\right)\left(-\eta_{s}^{2} \xi(\varphi), 1\right)\left(0, \frac{1}{r}\right)
\end{array}}\!\!\!&\!\!{ \zeta_{x i} C, \frac{A_{0} h_{a l}\left(q_{B k} \overline{\gamma_{H}}\right)^{\frac{1}{r}} }{\alpha \beta}}\end{array}\!\!\!\!\right]d\varphi.
\label{ABERIE}
    \end{align}
    \begin{proof}
See Appendix \ref{F}.
\end{proof}
\end{lemma}
\vspace{-2ex}
\noindent The average BER for OOK, M-QAM, and M-PSK modulations can then be expressed based on Lemma 5 as follows
\begin{align}\label{BERExp}
    \overline{{P}_{e}}=\delta_{B} \sum_{k=1}^{N_{B}} I\left(p_{B}, q_{B k}\right).
\end{align}
\noindent Moreover, a very tight asymptotic expression for the average BER in (\ref{BERExp}) may be derived when $\overline{\gamma_H}$ is large
by substituting (\ref{finalCDFA}) into (\ref{ABERI}) then employing $\int_0^{\infty} \gamma^{b-1} \exp (-a \gamma)d\gamma=\frac{ \Gamma(b)}{a^b}$ as shown by (\ref{ABEREA}).
\subsubsection{Ergodic Capacity}
The ergodic capacity of the end-to-end system where the FSO link is operating under either heterodyne or IM/DD techniques can be formulated as given in \cite[Eq.(26)]{lapidoth}, as follows
\addtocounter{equation}{1}
\begin{align}\label{DefC}
\overline{C}\triangleq \mathbb{E}[\ln(1+c_0\,\gamma)]=\int_{0}^{\infty}\ln(1+c_0\,\gamma)f_\gamma(\gamma)\,d\gamma,
\end{align}
where $c_0$ is a constant, taking the value $c_0=1$ for the heterodyne technique ($r=1$) and $c_0=e/(2\pi)$ for the IM/DD technique ($r=2$). Then, the bivariate Fox-H function can be used to formulate a unified equation for the ergodic capacity  applicable to both heterodyne and IM/DD types of detection as
\begin{align}
    \label{Capacity}
\nonumber &\overline{C}=\frac{ \eta_{s}^{2}}{2 \pi r q_{H} \Gamma(\alpha) \Gamma(\beta)} \sum_{i=1}^{N_{x}} \alpha_{x i}\zeta_{x i}^{-\beta_{x i}}\int_{-\pi}^{\pi}  {\rm H}_{1,0 ; 0,2 ; 4,3}^{0,1 ; 2,0 ; 1,4}\\&\footnotesize\left[\!\!\!\!\begin{array}{c|c}
\begin{matrix}
\left(1 ; 1, \frac{1}{r}\right) \\
-\\- \\
(0,1)\left(\beta_{x i}, 1\right) \\
\left(1-\eta_{s}^{2} \xi(\varphi), 1\right)(1-\alpha, 1)(1-\beta, 1)\left(1, \frac{1}{r}\right) \\
\left(1, \frac{1}{r}\right)\left(-\eta_{s}^{2} \xi(\varphi), 1\right)\left(0, \frac{1}{r}\right)
\end{matrix}\!\!\!&\!\!\zeta_{x i} C, \frac{A_{0} h_{a l}\left(c_{0} \overline{\gamma_{H}}\right)^{\frac{1}{r}} }{\alpha \beta}\end{array}\!\!\!\!\right]d\varphi.
\end{align}
\begin{proof}
See Appendix \ref{G}.
\end{proof}
\section{NUMERICAL ANALYSIS}
In this section, we demonstrate the mathematical formalism described above and verify its accuracy through Monte-Carlo simulations, employing the system settings outlined in Table \ref{tab2}. Analytical results are provided and compared with Monte-Carlo simulations. The comparison reveals a strong agreement between the derived analytical expressions and the simulated results, confirming the accuracy of the provided results.
\begin{table}[!h]
\caption{System Parameters}
\begin{tabular}{cccc}
\hline
Parameters          & Values                      & Parameters     & Values      \\ \hline
$H_O$               & 10 m                        & $F_0$          & $\infty$    \\
$H_U$               & 2 m                      & $\omega$       & 30 m/s      \\
$H_I$               & 20 m                        & $r_a$          & 5 mm        \\
  $H_H$   & 20 km   &           $\omega_b$     & 3$r_a$      \\
$d_{HI0}$           & 5 km                         & $d_{IU0}$      & 10 m        \\
V                   & 10 km                       & $\sigma_S$     & $r_a$       \\
$\lambda$           & 1550 nm                     & $q_V$          & 1.6         \\
$W_0$               & 1 mm                        & $N$            & 50          \\
$m_N$               & 1                           & $N_x$          & 75        \\
$\sigma_N^2$        & 2                           & $P_h$          & 0 dBm       \\
$L_{\text{Rain}}$   & 0.01 dB/km                  & $\sigma_U^2$   & $10^{-16}$  \\
$L_{\text{Atm}}$    & 5.4 $\times10^{-3}$ dB/km   & $C$            & 1           \\
$L_{\text{oth}}$    & 2 dB                        & $\gamma$       & 2 dB        \\
$G_{TX}$               & 50 dB                       &    $G_{RX}$          & 50 dB        \\
$P_{h0}$               & 0 dB                           & $\gamma_{th}$  & 2 dB\\$A$&   $1.7\times 10^{-13} \text{m}^{-2/3}$  &$f_c$ &  5 Ghz  \\ \hline
\multicolumn{4}{c}{Heavy shadowing (HS)  $b_R=0.063$ $m_R=1$ $\Omega_R=0.007$} \\
\multicolumn{4}{c}{Average showing (AS) $b_R=0.251$ $m_R=5$ $\Omega_R=0.279$}  \\
\multicolumn{4}{c}{Light shadowing (LS) $b_R=0.158$ $m_R=19$ $\Omega_R=1.29$}  \\ \hline
\end{tabular}\label{tab2}
\end{table}

First, we evaluate the accuracy of the mixture Gamma approximation of the RF channel and determine the optimal value for $N_x$. In this context, Fig.~\ref{fig2} depicts the OP of the RF link (HAP-IRS-user) under different shadowing conditions, for various values of $N_x$.
It can be clearly seen from the figure that as the shadowing conditions become more severe, the value of $N_x$ for achieving an accurate approximation decreases. Specifically, in the HS state, a
  value of 40 yields a highly accurate approximation, while in the AS scenario, 
  $N_x$ needs to be set to 50. In contrast, in the LS situation, $N_x$
  should be increased to 60 for an accurate approximation. 
Based on this analysis, to ensure more precise results, we select $N_x=75$ for this study.
\begin{figure}[!h]
\centering\includegraphics[scale=0.7, trim={0 0 0 0}, clip]{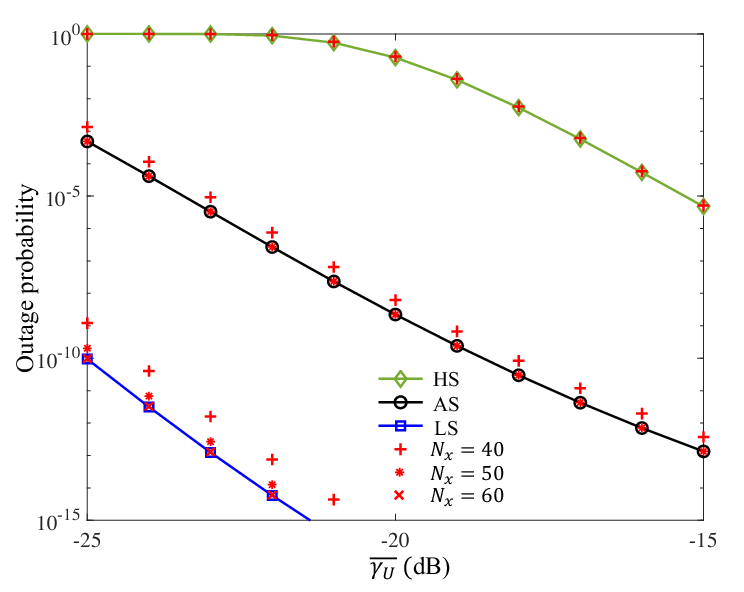}
    \caption{Outage probability of the RF channel under different shadowed conditions using different values of $N_x$ when $q_H=1$.}
    \label{fig2}
\end{figure}

 Fig.~\ref{Height2}  illustrates the outage probability performance under IM/DD and heterodyne detection schemes, for various altitudes of the HAP, denoted as $H_H$. As clearly observed from this figure, operating at a lower altitude improves the OP performance, regardless of the detection technique employed for the OGS-to-HAP FSO link.
This observation aligns with practical scenarios, as a higher HAP height increases the propagation distance between the OGS, HAP, and IRS, thereby impacting the communication link more significantly. 
For instance, under heterodyne detection technique and for an SNR of 40 dB, the system’s OP for HAP heights of 18 km, 20 km, and 22km are approximately $2.2 \times 10^{-2}$, $7.4 \times 10^{-3}$ , and  $2.4 \times 10^{-3}$, respectively.
\begin{figure}[!h]
\centering
\includegraphics[scale=0.75, trim={0 0 0 0}, clip]{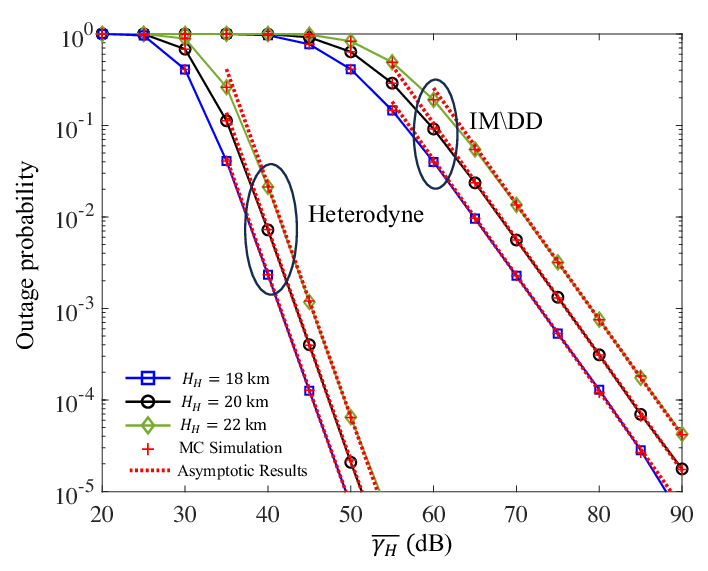}
    \caption {Outage probability for different \( H_H \) values under IM/DD and heterodyne detection techniques, considering LS conditions with \( q_H = 1 \) and \( \zeta = 40^\circ \).}
    \label{Height2}
\end{figure}

 \begin{figure}[!h]
\centering\includegraphics[scale=0.7, trim={5 0 0 0}, clip]{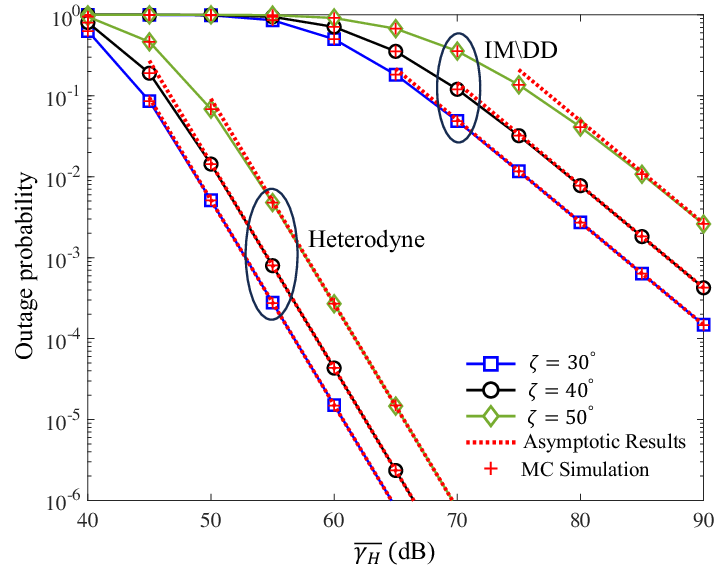}
    \caption{Outage probability under IM/DD and heterodyne detection techniques for different zenith angles under HS condition with $q_H=1$.}
    \label{fig3}
\end{figure}

In Fig.~\ref{fig3}, the OP is illustrated versus the average SNR of the FSO link for different zenith angle values $\zeta$ using both heterodyne and IM/DD techniques. 
We can clearly observe from this figure that regardless of whether we employ IM/DD or heterodyne detection, minimizing the zenith angle results in lower outage probability and improved system performance. This highlights the significance of selecting smaller zenith angles to optimize the reliability and effectiveness of the end-to-end communication system. In addition, the figure clearly shows a perfect agreement between the exact and asymptotic results, particularly at high SNR values. This convergence highlights the accuracy of our asymptotic analysis, demonstrating its effectiveness under both IM/DD and heterodyne detection schemes.
\begin{figure}[H]
\centering\includegraphics[scale=0.7, trim={0 0 0 0}, clip]{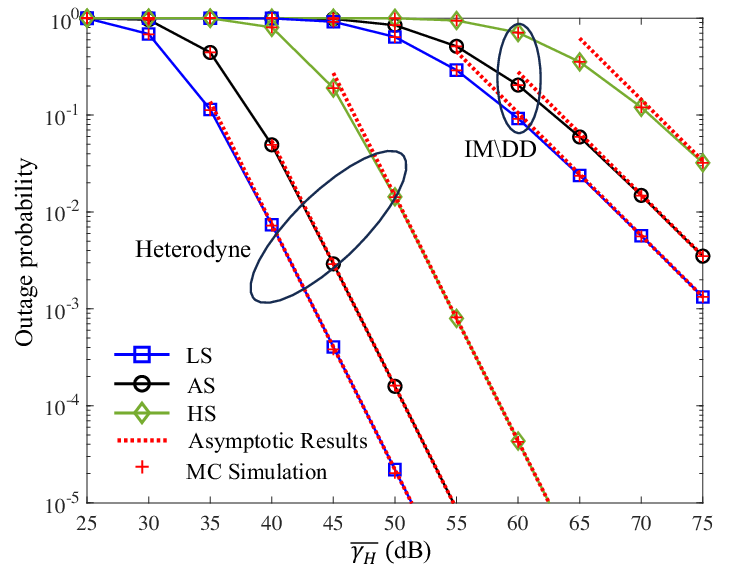}
    \caption{Outage probability for different shadowed conditions of the RF link under IM/DD and heterodyne techniques with $q_H=1$.}
\label{fig4}
\end{figure}

To reflect the impact of the RF channel shadowing conditions on the end-to-end outage performance, Fig.~\ref{fig4} depicts the OP under heavy (HS), average (AS), and light (LS) shadowing conditions, for both heterodyne and IM/DD schemes. Clearly, we can see that under less severe shadowing, the system performs better, irrespective of the chosen detection method for the FSO link. 

\begin{figure}[!h]
\centering\includegraphics[scale=0.7, trim={0 0 0 0}, clip]{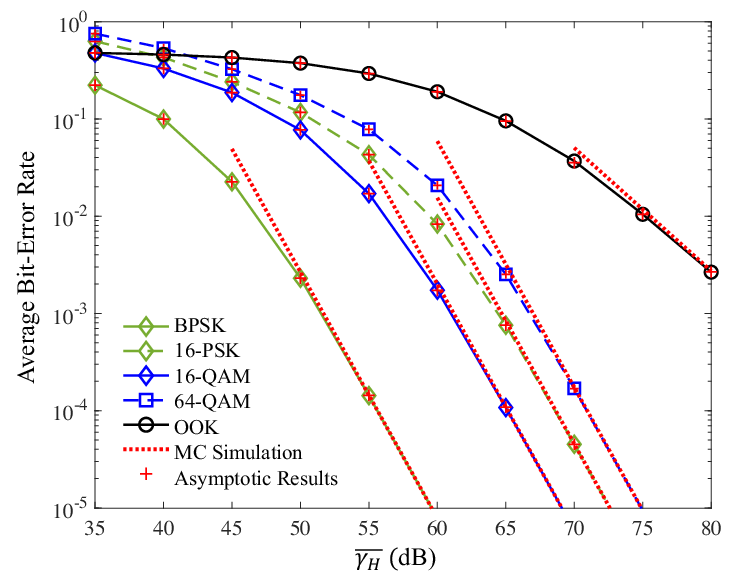}
    \caption{Average BER using different IM/DD and coherent modulation schemes under HS condition with $q_H=1$.}
    \label{fig5}
\end{figure}
Fig.~\ref{fig5} presents the average BER for various modulation schemes such as OOK for IM/DD technique as well as BPSK, 16-PSK, 16-QAM, and 64-QAM for heterodyne technique. 
Clearly, this figure illustrates that the application of the heterodyne technique significantly enhances the BER performance for all modulation schemes. In addition, as expected, Fig.~\ref{fig5} shows that 16-QAM performs better than 16-PSK. Moreover, among the modulation techniques provided, BPSK modulation stands out for its superior performance. Furthermore, it is clear that the asymptotic BER results converge to the exact results at high SNR, confirming the accuracy of the asymptotic analysis.

\begin{figure}[!h]
\centering\includegraphics[scale=0.7, trim={0 0 0 0}, clip]{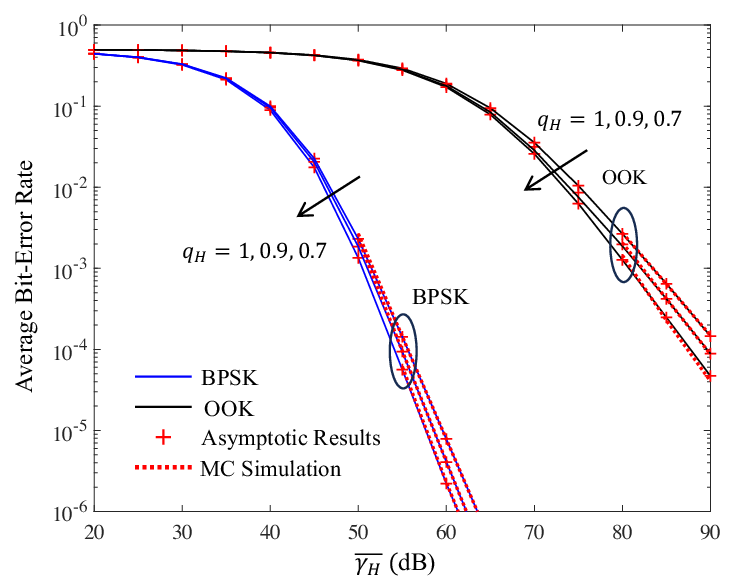}
    \caption{Average BER using both BPSK and OOK modulation techniques for different
ratios of vertical and horizontal beam deviations under HS condition.}
    \label{fig6}
\end{figure}
Fig.~\ref{fig6} illustrates the average BER under various $q_H$ values, for both OOK and BPSK modulation schemes. From Fig.~\ref{fig6}, it is apparent that when the beam orientation deviation exhibits asymmetrical behavior, the end-to-end communication system performs better. For instance, for an average SNR of $\overline{\gamma_H}=$55 dB, the average BER for BPSK modulation is approximately $\overline{P_e}=1.4\times 10^{-4}$ when horizontal and vertical orientation deviations are equal ($q_H = 1$). However, the average BER reduces to $\overline{P_e}=6.1\times 10^{-5}$ when the ratio of horizontal to vertical beam deviations is $q_H = 0.7$. Other noteworthy results are also obtained, which are consistent with the patterns seen in Fig.~\ref{fig5} and include the advantage of BPSK modulation over IM/DD as well as the high accuracy of the asymptotic BER result at high SNR regime. 
 \begin{figure}[!h]
\centering\includegraphics[scale=0.7, trim={0 0 0 0}, clip]{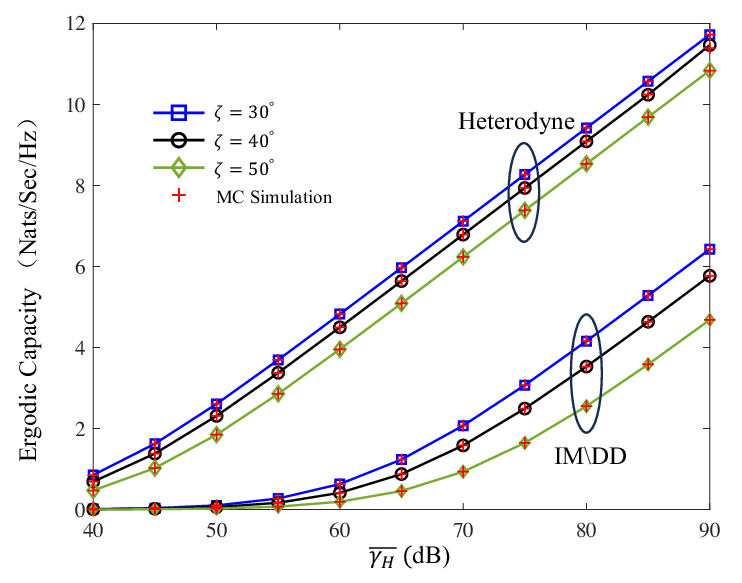}
    \caption{Ergodic capacity for different levels of the zenith angle with $q_H=1$ using both types of detection.}
    \label{fig7}
\end{figure}

The benefit of a smaller zenith angle on the ergodic capacity is shown in Fig.~\ref{fig7}. It is evident from Fig.~\ref{fig7} that the ergodic capacity performance improves as the zenith angle decreases, regardless of whether IM/DD or heterodyne techniques are employed. Furthermore, it is noticeable from Fig.~\ref{fig8} that under more severe shadowing conditions (HS), the end-to-end system performance declines for both types of detection methods. 
\begin{figure}[H]
\centering\includegraphics[scale=0.7, trim={0 0 0 0}, clip]{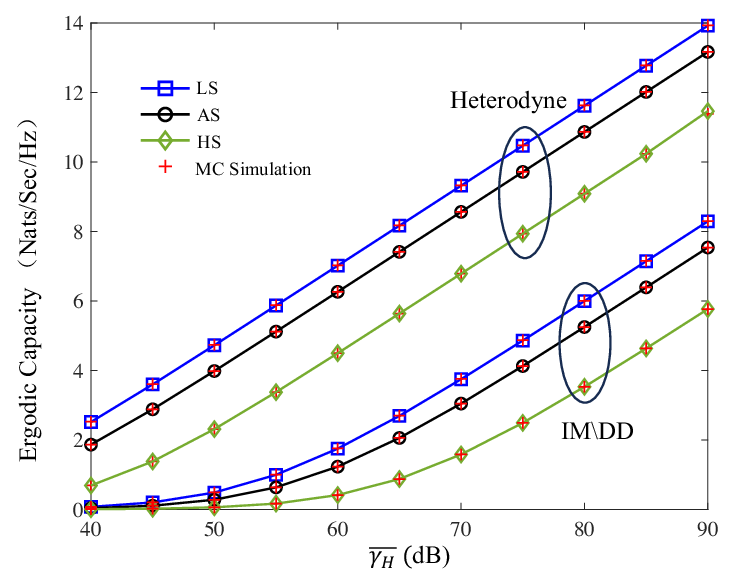}
    \caption{Ergodic capacity under IM/DD and heterodyne techniques for different shadowing condition with $q_H=1$.}
    \label{fig8}
\end{figure}
Moreover, in line with with the earlier analysis of outage performance, the heterodyne technique consistently outperforms IM/DD under all shadowing conditions.

\section{Conclusion}
This paper examines the performance of a ground-to-HAP-to-IRS-to-users system using RF technology for the HAP-to-IRS-to-users link and FSO technology for the ground-to-HAP link in terms of the average BER, the ergodic capacity, and the outage probability when the FSO link uses either heterodyne or IM/DD techniques. Utilizing the Gamma-Gamma distribution with generalized pointing errors for the FSO link, and approximating the cascaded shadowed-rician and Nakagami-m distributed RF link with the mixture Gamma model, analytical expressions for the aforementioned performance metrics are obtained in terms of the bivariate Fox-H function. Furthermore, asymptotic results are presented in terms of simple functions for the outage probability and the average BER in the high SNR regime.
The numerical results that have been provided have clearly shown how the system performance is affected by pointing errors, shadowing conditions, and the zenith angle. Our analysis has demonstrated that increasing the zenith angle or intensifying the shadowing effect of the RF link can lead to a significant degradation in the overall system performance. Furthermore, pointing errors can severely impair the end-to-end performance, particularly when vertical and horizontal beam deviations are equal. In addition, heterodyne detection demonstrates its effectiveness in enhancing the end-to-end system performance by increasing the ergodic capacity and reducing the average BER and the outage probability.
\appendices
\numberwithin{equation}{section}%

\section{PDF of $\gamma_U$}
\label{A}
In this appendix, we calculate the PDF of $\gamma_U=\overline{\gamma_U} Z^{2}$, where $Z=\left|\sum_{i=1}^{N} \alpha_{si}\beta_{si}\right|$. For sufficiently large values of $N$, according to the central limit theorem (CLT), the variable $Z$ follows a normal distribution.
The random variable $Z^2$ is then distributed according to a non-central chi-square distribution having one degree of freedom with mean $\mu_Z=N\mathbb{E}({\alpha_{si}})\mathbb{E}({\beta_{si}})$, variance ${\sigma_Z}^2=N\mathbb{E}({\alpha_{si}}^2)\mathbb{E}({\beta_{si}}^2)-{\mu_Z}^2$, which can be expressed as
\begin{align}
\label{A1}
\nonumber & f_{Z^2}\left(Z^2\right)=\frac{\left(\ {\mu_{Z}}^{2}\right)^{\frac{1}{4}}}{2 {\sigma_{Z}}^{2}} \exp \left(-\frac{ {\mu_{Z}}^{2}}{2 {\sigma_{Z}}^{2}}\right)\\&
\times {(Z^2)}^{-\frac{1}{4}} \exp \left(-\frac{Z^2}{2 {\sigma_{Z}}^{2}}\right) I_{-\frac{1}{2}}\left(\frac{\mu_{Z}}{{\sigma_{Z}}^{2}} \sqrt{Z^2}\right),
\end{align}
where $I_v(\cdot)$ stands for the $v$th-order modified Bessel function of the first kind.
Utilizing $\gamma_U=\overline{\gamma_U} Z^{2}$ and (\ref{A1}), the PDF of $\gamma_U$ can be given in (\ref{PDFU}). 
\section{Mixture Gamma Approximation}
\label{B}
Using \cite[Eq.~(8.445)]{intetable}, the PDF expression of $\gamma_{U}$ in (\ref{PDFU}) can be re-written as

\begin{align}\label{B1}
   \nonumber &f_{\gamma_{U}}(x)=\sum_{i=1}^{\infty} \frac{\left(\overline{\gamma_{U}} \mu_{\gamma_{U}}^{2}\right)^{\frac{1}{4}}}{2 \overline{\gamma_{U}} \sigma_{\gamma_{U}}{ }^{2}} \exp \left(-\frac{\mu_{\gamma_{U}}{ }^{2}}{2{\sigma_{\gamma_{U}}}^{2}}\right)\\
   & \times \frac{\left(\frac{\mu_{\gamma_{U}}}{2 \sigma_{\gamma_{U}}{ }^{2}} \sqrt{\frac{1}{\bar{\gamma}_{U}}}\right)^{-\frac{5}{2}+2 i}}{(i-1) ! \Gamma\left(i-\frac{1}{2}\right)}x^{-\frac{1}{2}+i-1} \exp \left(-\frac{x}{2 \overline{\gamma_{U}} \sigma_{\gamma_{U}}{ }^{2}}\right).
\end{align}
Now by applying \cite[Eq.~(1)]{mixturegamma}, we can approximate $f_{\gamma_{U}}$ using the mixture Gamma model as shown by (\ref{PDFUa}) with parameters obtained by matching the two PDFs in (\ref{PDFU}) and (\ref{B1}) as given in (\ref{PDFparam}).
\section{CDF and PDF of the End-to-End SNR}
\label{C}
\vspace{-1ex}
The CDF of the end-to end SNR $\gamma$ in (\ref{TSNR}) can be formulated as
\begin{align}
\label{C1}
    F_{\gamma}(\gamma)=\int_{0}^{\infty} F_{\gamma_{H}}\left(\gamma\left(1+\frac{C}{x}\right)\right) f_{\gamma_{U}}(x)\, dx.
\end{align}
Substituting (\ref{CDFH}) and (\ref{PDFUa}) into (\ref{C1}) and employing \cite[Eq.~(11)]{MeijerGalgorithm}, $F_{\gamma}(\gamma)$ can be written as
\begin{align}\label{C2}
\nonumber &F_{\gamma}(\gamma)=1-  \frac{ \eta_{s}^{2} }{2 \pi q_{H} \Gamma(\alpha) \Gamma(\beta)}\int_{-\pi}^{\pi} \sum_{i=1}^{N_{x}} \alpha_{x i} \int_{0}^{\infty} \\
\nonumber & \times {\rm G}_{2,4}^{4,0}\left[\frac{\alpha \beta}{A_{0} h_{a l}h_{af}}\left(\frac{\gamma}{\overline{\gamma_{H}}}\right)^{\frac{1}{r}}\left(1+\frac{C}{x}\right)^{\frac{1}{r}}  \Bigg | \begin{array}{c}
1+\eta_{s}^{2} \xi(\varphi), 1 \\
0, \eta_{s}^{2} \xi(\varphi), \alpha, \beta
\end{array}\right] \\
&\times  x^{\beta_{x i}-1} {\rm G}_{0,1}^{1,0}\left[\zeta_{x i} x \bigg| \begin{matrix} - \\
			0
		\end{matrix} \right]
 dx\, d\varphi.
\end{align}
Using the Meijer-G function's primary definition in \cite[Eq.~(9.301)]{intetable}, (\ref{C2}) can be expressed as
\begin{align}
\label{C3}
   \nonumber  &F_{\gamma}(\gamma)=1-\frac{ \eta_{s}^{2}}{2 \pi q_{H} \Gamma(\alpha) \Gamma(\beta)} \int_{-\pi}^{\pi} \sum_{i=1}^{N_{x}} \alpha_{x i} \\
  \nonumber & \times \frac{1}{(2 \pi \mathrm{i})^{2}} \int \frac{\Gamma\left(\eta_{s}^{2} \xi(\varphi)-t\right) \Gamma(\alpha-t) \Gamma(\beta-t)}{\mathrm{t} \Gamma\left(1+\eta_{s}^{2} \xi(\varphi)-t\right)}\\
  \nonumber &\times \left(\frac{\alpha \beta}{A_{0} h_{a l}}\right)^{t}\left(\frac{\gamma}{\overline{\gamma_{H}}}\right)^{\frac{t}{r}} \mathrm{~d} t \int \Gamma(-s)\left(\zeta_{x i}\right)^{s} ds \\
   &\times \int_{x=0}^{\infty}\left(1+\frac{C}{x}\right)^{\frac{t}{r}} x^{\beta_{x i}-1+s}\, dx\,d\varphi.
\end{align}
Then, by utilizing \cite[Eq.(3.251/11)]{intetable}
\begin{align}
\int_{0}^{\infty} x^{\mu-1}\left(1+\beta x^{p}\right)^{-\nu} d x=\frac{1}{p} \beta^{-\frac{\mu}{p}} \mathrm{~B}\left(\frac{\mu}{p}, \nu-\frac{\mu}{p}\right),
\end{align} 
where $\mathrm{B}(x, y)=\frac{\Gamma(x) \Gamma(y)}{\Gamma(x+y)}$ represents the Beta function \cite[Eq.~(8.384/1)]{intetable}.
Then,  (\ref{C3}) can be re-written in the form
\begin{align}
\label{C5}
\nonumber & F_{\gamma}(\gamma)=1-\frac{ \eta_{s}^{2} }{2 \pi q_{H} r\Gamma(\alpha) \Gamma(\beta)}\sum_{i=1}^{N_{x}} \alpha_{x i}  \int_{-\pi}^{\pi}C^{\beta_{x i}}\frac{1}{(2 \pi \mathrm{i})^{2}}  \\
\nonumber & \times
\iint\left(\frac{\alpha \beta}{A_{0} h_{a l}h_{af}}\right)^{t}
\left(\frac{\gamma}{\overline{\gamma_{H}}}\right)^{\frac{t}{r}}\left(C \zeta_{x i}\right)^{s} \Gamma\left(\beta_{x i}+s-\frac{t}{r}\right)\\
&\frac{\Gamma\left(\eta_{s}^{2} \xi(\varphi)-t\right) \Gamma(\alpha-t) \Gamma(\beta-t) }{\Gamma\left(1+\eta_{s}^{2} \xi(\varphi)-t\right) \Gamma\left(1-\frac{t}{r}\right)}  \Gamma\left(-\beta_{x i}-s\right) \Gamma(-s) ds dt d\varphi.
\end{align}
Let $s^{\prime}=s+\beta_{x i}, t^{\prime}=-t$, then (\ref{C5}) can be presented as
\begin{align}
\label{C6}
\nonumber &F_{\gamma}(\gamma)=1-\frac{ \eta_{s}^{2} }{2 \pi q_{H} r \Gamma(\alpha) \Gamma(\beta)} \int_{-\pi}^{\pi} \sum_{i=1}^{N_{x}} \alpha_{x i} \zeta_{x i}^{-\beta_{x i}} \frac{1}{(2 \pi \mathrm{i})^{2}}\\
\nonumber & \times \iint \Gamma\left(s^{\prime}+\frac{t^{\prime}}{r}\right) \Gamma\left(-s^{\prime}\right) \Gamma\left(\beta_{x i}\right. 
\left.-s^{\prime}\right)C \zeta_{x i}^{s^{\prime}}\\
\nonumber & \times \frac{\Gamma\left(\eta_{s}^{2} \xi(\varphi)+t^{\prime}\right) \Gamma\left(\alpha+t^{\prime}\right) \Gamma\left(\beta+t^{\prime}\right)}{\Gamma\left(1+\eta_{s}^{2} \xi(\varphi)+t^{\prime}\right) \Gamma\left(1+\frac{t^{\prime}}{r}\right)}\\
&\times \left[\frac{A_{0} h_{a l}h_{af}}{\alpha \beta}\left(\frac{\gamma_{H}}{\gamma}\right)^{\frac{1}{r}}\right]^{t^{\prime}} d s^{\prime} d t^{\prime} d \varphi.
\end{align}
Finally, by using \cite[Eq.~(1.1)]{Mittal}, we can easily obtain the desired CDF expression in (\ref{finalCDF}).

Now, by differentiating (\ref{C6}) with respect to $\gamma$, we can compute the PDF of $\gamma$ as
\begin{align}\label{C7}
\nonumber &f_{\gamma}(\gamma)=\frac{ \eta_{s}^{2}}{2 \pi q_{H} r \Gamma(\alpha) \Gamma(\beta)\gamma}\int_{-\pi}^{\pi} \sum_{i=1}^{N_{x}} \alpha_{x i} \zeta_{x i}^{-\beta_{x i}} \frac{1}{(2 \pi \mathrm{i})^{2}} \\
\nonumber& \times
 \iint \Gamma\left(s^{\prime}+\frac{t^{\prime}}{r}\right) \Gamma\left(-s^{\prime}\right) \Gamma\left(\beta_{x i}-s^{\prime}\right)\left(C \zeta_{x i}\right)^{s^{\prime}} \\
 \nonumber& \times
 \frac{\Gamma\left(\eta_{s}^{2} \xi(\varphi)+t^{\prime}\right) \Gamma\left(\alpha+t^{\prime}\right) \Gamma\left(\beta+t^{\prime}\right)}{\Gamma\left(1+\eta_{s}^{2} \xi(\varphi)+t^{\prime}\right) \Gamma\left(\frac{t^{\prime}}{r}\right)}\\
 &\times\left[\frac{A_{0} h_{a l}h_{af}}{\alpha \beta}\left(\frac{\bar{\gamma}_{H}}{\gamma}\right)^{\frac{1}{r}}\right]^{t^{\prime}} ds^{\prime} d t^{\prime} d \varphi,
\end{align}
which can be calculated using the bivariate Fox-H function after applying \cite[Eq.(1.1)]{Mittal} as outlined in (\ref{finalPDF}).

\section{Moments}
\label{D}
Using \cite[Eq.~(2.3)]{Mittal}, the PDF of $\gamma$ in (\ref{PDFH}) can be re-written as
\begin{align}
\label{D1}
\nonumber &f_\gamma(\gamma) 
 =\frac{\eta_s^2\gamma^{-1}}{2 \pi q_H r \Gamma(\alpha) \Gamma(\beta)}  \int_{-\pi}^\pi \sum_{i=1}^{N_x} \alpha_{x i} \zeta_{x i}^{-\beta_{x i}}\\
 \nonumber & \times \int_0^{\infty}x^{-1} \exp(-x) {\rm H}_{0,2}^{2,0}\left[\zeta_{x i} C x\bigg| \begin{matrix} - \\(0,1), (\beta_{x i},1) \end{matrix}\right] \\
& \times {\rm H}_{3,2}^{0,3}\left[\!\!\!\!
\begin{array}{c|c}
\footnotesize\begin{matrix}
\left(1-\eta_s^2 \xi(\varphi), 1\right)(1-\alpha, 1)(1-\beta, 1) \\
\left(-\eta_s^2 \xi(\varphi), 1\right)\left(1, \frac{1}{r}\right)
\end{matrix}\!\!\!&\!\!\!  \frac{A_0 h_{a l}\left(\frac{\overline{\gamma_H}}{\gamma}x\right)^{\frac{1}{r}}}{\alpha \beta} \end{array}\!\!\!\!\right] dx d\varphi,
\end{align}
where $\rm{H}_{\cdot,\cdot}^{\cdot,\cdot}[\cdot]$ refers to the Fox-H function \cite[Eq.~(2.9.1)]{Htran1}.

\noindent Now, the moments formulated as,  $\mathbb{E}\left(\gamma^s\right)=\int_0^{\infty} \gamma^{s} f_\gamma(\gamma) d \gamma$, can be computed by substituting (\ref{D1}) into this formula, yielding

\begin{align}
\nonumber &\mathbb{E}\left(\gamma^s\right)=\frac{\eta_s^2}{2 \pi q_H \Gamma(\alpha) \Gamma(\beta)} \int_{-\pi}^\pi \sum_{i=1}^{N_x} \alpha_{x i} \frac{\zeta_{x i}^{-\beta_{x i}}}{r}\int_0^{\infty}   x^{-1}\\
\nonumber & \times \exp(-x) {\rm H}_{0,2}^{2,0}\left[\zeta_{x i} C x\bigg| \begin{matrix} - \\(0,1), (\beta_{x i},1) \end{matrix}\right]
 \int_0^{\infty} \gamma^{s-1}  \\
 &\times{ {\rm H}_{3,2}^{0,3}\!\!\left[\!\!\!\!\footnotesize\begin{array}{c|c}
\begin{matrix}
\left(1-\eta_s^2 \xi(\varphi), 1\right)(1-\alpha, 1)(1-\beta, 1) \\
\left(-\eta_s^2 \xi(\varphi), 1\right)\left(1, \frac{1}{r}\right)
\end{matrix}\!\!\!&\!\!\!\frac{A_0 h_{a l}\left(\frac{\overline{\gamma_H}}{\gamma}x\right)^{\frac{1}{r}}}{\alpha \beta} \end{array}\!\!\!\!\right]}\normalsize d\gamma dx d\varphi.
\end{align}
Utilizing \cite[Eq.~(1.59)]{Htran2} and \cite[Eq.~(2.8)]{Htran2}, the moments of $\gamma$ can be evaluated according to (\ref{MomentsH}).

\section{Asymptotic Result}
\label{E}

For high values of $\overline{\gamma_{H}}$, the following Meijer-G function can be
approximated by using \cite[Eq.~(1.8.4)]{Htran1} as

\begin{align}
\label{E1}
   \nonumber  &{\rm G}_{2,4}^{4,0}\left[\frac{\alpha \beta}{A_{0} h_{a l}}\left(\frac{\gamma}{\overline{\gamma_{H}}}\right)^{\frac{1}{r}}\left(1+\frac{C}{x}\right)^{\frac{1}{r}}  \begin{array}{|c}
1+\eta_{s}^{2} \xi(\varphi), 1 \\
0, \eta_{s}^{2} \xi(\varphi), \alpha, \beta
\end{array}\right]\\
&\underset{\overline{\gamma_H}\gg 1}{\mathop{\approx }}\sum_{j}^{4} {h}_{j}\left[\frac{\alpha \beta}{A_{0} h_{a l}}\left(\frac{\gamma}{\overline{\gamma_{H}}}\right)^{\frac{1}{r}}\left(1+\frac{C}{x}\right)^{\frac{1}{r}}\right]^{\mathcal{H}_j},
\end{align}
where ${h}_j=\left\{\eta_{s}^{2} \xi(\varphi),\alpha,\beta,0\right\}$, and

\begin{align}
    \begin{cases}
 \mathcal{H}_{1}=-\frac{\Gamma\left(\alpha-\eta_{s}^{2} \xi(\varphi)\right) \Gamma\left(\beta-\eta_{s}^{2} \xi(\varphi)\right)}{\eta_{s}^{2} \xi(\varphi)},
\\[1.5ex]
\mathcal{H}_{2}=-\frac{\Gamma(\beta-\alpha)}{\alpha\left(\eta_{s}^{2} \xi(\varphi)-\alpha\right)},\\[1.5ex]
\mathcal{H}_{3}=-\frac{\Gamma(\alpha-\beta)}{\beta\left(\eta_{s}^{2} \xi(\varphi)-\beta\right)},
 \\[1.5ex]
 \mathcal{H}_{4}=\frac{\Gamma(\alpha) \Gamma(\beta)}{\eta_{s}^{2} \xi(\varphi)}.
\end{cases}
\end{align}
Substituting (\ref{E1}) into (\ref{C2}) and applying \cite[Eq.~(9.301)]{intetable}, the CDF expression in (\ref{finalCDF}) can be re-written as

\begin{align}
  \nonumber  &F_{\gamma}(\gamma)\underset{\overline{\gamma_H}\gg 1}{\mathop{\approx }}1-\frac{\eta_{s}^{2}  }{2 \pi q_{H} \Gamma(\alpha) \Gamma(\beta)} \int_{-\pi}^{\pi} \sum_{i=1}^{N_{x}} \alpha_{x i} \\
  \nonumber &\times \sum_{j}^{4} \int_{0}^{\infty} h_{j}\left[\frac{\alpha \beta}{A_{0} h_{a l}h_{af}}\left(\frac{\gamma}{\overline{\gamma_{H}}}\right)^{\frac{1}{r}}\left(1+\frac{C}{x}\right)^{\frac{1}{r}}\right]^{\mathcal{H} j}\\
  & \times x^{\beta_{x i}-1}\frac{1}{2 \pi \mathrm{i}} \int \Gamma(-s)\left(\zeta_{x i} x\right)^{s} ds\, dx\, d\varphi.
\end{align}
By utilizing the integral in \cite[Eq.(3.194/3)]{intetable}

\begin{equation}
    \int_0^{\infty} \frac{x^{\mu-1} d x}{(1+\beta x)^\nu}=\beta^{-\mu} \mathrm{B}(\mu, \nu-\mu),
\end{equation} 
where $\mathrm{B}(x, y)=\frac{\Gamma(x) \Gamma(y)}{\Gamma(x+y)}$ represents the Beta function \cite[Eq.~(8.384/1)]{intetable}, we can obtain the asymptotic equation of the CDF in (\ref{finalCDFA}).
\section{Average BER}
\label{F}
Substituting (\ref{C6}) into (\ref{ABERI}) yields 

\begin{align}\label{H1}
\nonumber &I\left(p_B, q_{B k}\right)=\frac{1}{2}-  \frac{q_{B k}{ }^{p_B}}{2 \Gamma\left(p_B\right)} \int_0^{\infty} \gamma^{p_B-1-\frac{t^{\prime}}{r}} \exp \left(-q_{B k} \gamma\right)d \gamma\\
\nonumber &  \times \frac{\eta_s^2}{2 \pi q_H r\Gamma(\alpha) \Gamma(\beta)}   \int_{-\pi}^\pi \sum_{i=1}^{N_x} \alpha_{x i} \zeta_{x i}^{-\beta_{x i}} \frac{1}{(2 \pi \mathrm{i})^2} \iint \\
\nonumber & \times \Gamma\left(s^{\prime}+\frac{t^{\prime}}{r}\right) \Gamma\left(-s^{\prime}\right) \Gamma\left(\beta_{x i}-s^{\prime}\right)\left(C \zeta_{x i}\right)^{s^{\prime}} \\
\nonumber & \times \frac{\Gamma\left(\eta_s^2 \xi(\varphi)+t^{\prime}\right) \Gamma\left(\alpha+t^{\prime}\right) \Gamma\left(\beta+t^{\prime}\right)}{\Gamma\left(1+\eta_s^2 \xi(\varphi)+t^{\prime}\right) \Gamma\left(1+\frac{t^{\prime}}{r}\right)}\\
& \times \left[\frac{A_0 h_{a l}}{\alpha \beta}\left(\overline{\gamma_H}\right)^{\frac{1}{r}}\right]^{t^{\prime}} ds^{\prime} dt^{\prime} d\varphi.
\end{align}
Using the definition of Gamma Function $\Gamma(z)$ in \cite{intetable}, $\Gamma(z)=\int_0^{\infty} t^{z-1} e^{-t} d t$, (\ref{H1}) can be expressed as

\begin{align}\label{F2}
\nonumber &I\left(p_B, q_{B k}\right)=\frac{1}{2} - \frac{\eta_s^2}{4 \pi q_H r \Gamma(\alpha) \Gamma(\beta)\Gamma\left(p_B\right)} \\
\nonumber & \times \int_{-\pi}^\pi \sum_{i=1}^{N_x} \alpha_{x i} \zeta_{x i}^{-\beta_{x i}} \frac{1}{(2 \pi \mathrm{i})^2}   \iint \Gamma\left(s^{\prime}+\frac{t^{\prime}}{r}\right) \\
\nonumber & \times \Gamma\left(-s^{\prime}\right) \Gamma\left(\beta_{x i}-s^{\prime}\right)\left(C \zeta_{x i}\right)^{s^{\prime}} \\
\nonumber & \times \frac{\Gamma\left(\eta_s^2 \xi(\varphi)+t^{\prime}\right) \Gamma\left(\alpha+t^{\prime}\right) \Gamma\left(\beta+t^{\prime}\right) \Gamma\left(p_B-\frac{t^{\prime}}{r}\right)}{\Gamma\left(1+\eta_s^2 \xi(\varphi)+t^{\prime}\right) \Gamma\left(1+\frac{t^{\prime}}{r}\right)}\\
&\times\left[\frac{A_0 h_{a l}}{\alpha \beta}\left(q_{B k} \bar{\gamma}_H\right)^{\frac{1}{r}}\right]^{t^{\prime}} ds^{\prime} dt^{\prime} d\varphi.
\end{align}
Finally, we can derive (\ref{ABERIE}) by applying \cite[Eq.~(1.1)]{Mittal} to (\ref{F2}).

\section{Ergodic Capacity}
\label{G}
Substituting (\ref{C7}) into (\ref{DefC}) results in
\begin{align}
\nonumber & \overline{C}=\frac{\eta_s^2}{2 \pi q_H r\Gamma(\alpha) \Gamma(\beta)}  \int_{-\pi}^\pi \sum_{i=1}^{N_x} \alpha_{x i}\zeta_{x i}^{-\beta_{x i}} \frac{1}{(2 \pi \mathrm{i})^2}\\
\nonumber & \times \iint \Gamma\left(s^{\prime}+\frac{t^{\prime}}{r}\right) \Gamma\left(-s^{\prime}\right) 
\Gamma\left(\beta_{x i}-s^{\prime}\right)\left(C \zeta_{x i}\right)^{s^{\prime}}\Gamma\left(\frac{t^{\prime}}{r}\right)  \\
\nonumber & \times\frac{\Gamma\left(\eta_s^2 \xi(\varphi)+t^{\prime}\right)\Gamma\left(\alpha+t^{\prime}\right) \Gamma\left(\beta+t^{\prime}\right)}{\Gamma\left(1+\eta_s^2 \xi(\varphi)+t^{\prime}\right) \Gamma\left(1+\frac{t^{\prime}}{r}\right)}  \\
&\times\Gamma\left(1-\frac{t^{\prime}}{r}\right)\left[\frac{A_0 h_{a l}}{\alpha \beta}\left(c_0 \overline{\gamma_H}\right)^{\frac{1}{r}}\right]^{t^{\prime}} d s^{\prime} d t^{\prime} d \varphi.
\end{align}
Using the integral identity $\int_0^{\infty} x^{\mu-1} \ln (1+\gamma x) d x=\frac{\pi}{\mu \gamma^\mu \sin \mu \pi}$ in \cite[Eq.~(4.293/10)]{intetable} along with  $\Gamma(z) \Gamma(1-z)=\frac{\pi}{\sin \pi z}$ in \cite[Eq.(2), p99]{specialfunc}, and applying (1.1) of \cite{Mittal}, the ergodic capacity can be derived in (\ref{Capacity}).

\bibliographystyle{IEEEtran}
\bibliography{IEEEexample}

\end{document}